\documentclass[11pt,a4paper]{article}
\usepackage[T1]{fontenc}
\usepackage{lmodern}
\usepackage{xcolor}
\usepackage[margin=28mm]{geometry}
\usepackage{amsmath,amssymb,amsthm,mathtools}
\usepackage{mathrsfs}
\usepackage{authblk}
\usepackage{booktabs,array,tabularx,microtype}
\usepackage{enumitem,placeins}
\usepackage[colorlinks=true,linkcolor=blue,citecolor=blue,urlcolor=blue]{hyperref}
\allowdisplaybreaks[2]
\hypersetup{pdfauthor={Gabor P. Nagy and Yue Zhou},pdftitle={Semifields in prime dimensions and counterexamples to Kaplansky's conjecture},pdfkeywords={semifields, isotopy, norm determinants, vertex ranks, MRD codes}}
\setlist[enumerate]{label=\textup{(\roman*)},itemsep=3pt,topsep=5pt}
\newtheorem{theorem}{Theorem}[section]
\newtheorem{proposition}[theorem]{Proposition}
\newtheorem{lemma}[theorem]{Lemma}
\newtheorem{corollary}[theorem]{Corollary}
\theoremstyle{definition}
\newtheorem{definition}[theorem]{Definition}
\newtheorem{remark}[theorem]{Remark}
\newcommand{\F}{\mathbb F}
\newcommand{\Ps}{\mathcal P}
\newcommand{\Qs}{\mathcal Q}
\newcommand{\Ss}{\mathcal S}
\newcommand{\Ds}{\mathcal D}

\newcommand{\Fs}{\mathcal F}
\newcommand{\PG}{\operatorname{PG}}
\newcommand{\N}{\operatorname{N}}
\newcommand{\Tr}{\operatorname{Tr}}
\newcommand{\rank}{\operatorname{rank}}
\newcommand{\diag}{\operatorname{diag}}
\newcommand{\id}{\operatorname{id}}

\newcommand{\End}{\operatorname{End}}

\newcommand{\Z}{Z}
\newcommand{\Nl}{N_l}
\newcommand{\Nm}{N_m}
\newcommand{\Nr}{N_r}
\newcommand{\multiset}[1]{\{\!\{#1\}\!\}}
\newcommand{\Bf}{\boldsymbol}
\DeclareMathOperator{\Gal}{Gal}
\DeclareMathOperator{\Aut}{Aut}

\title{Semifields in prime dimensions and counterexamples to Kaplansky's conjecture}

\author[1]{G\'abor P. Nagy}
\author[2]{Yue Zhou}
\affil[1]{Bolyai Institute, University of Szeged, Aradi v\'ertan\'uk tere 1, H-6720 Szeged, Hungary}
\affil[2]{College of Science, National University of Defense Technology,
Changsha 410073, China}
\date{}
\begin{document}
\maketitle
\begin{abstract}
In 1975, Kaplansky conjectured that every five-dimensional division
algebra over a sufficiently large finite field is a field or a twisted
field. We disprove this conjecture. For every prime power
$q=p^e\equiv1\pmod3$ and every $n\ge5$ with $\gcd(n,6)=1$, we construct
semifields of order $q^n$. For fixed $q,n$, the family represents
$\varphi(n)$ isotopy classes if $p\equiv1\pmod3$, and $\varphi(n)/2$
if $p\equiv2\pmod3$, where $\varphi$ is Euler's totient function. 
Using new isotopy invariants and the structural properties of our
construction, we prove that none of these semifields is isotopic to a
finite field or an Albert's generalized twisted field.
In particular, the five-dimensional specialization gives
infinitely many pairwise nonisotopic counterexamples to Kaplansky's
conjecture over arbitrarily large finite fields. More generally,
for each prime dimension $n\ge5$, the examples occur over arbitrarily
large fields in every characteristic other than three, contradicting
the classification asserted by Menichetti in 1996, whose proof contains gaps.
\end{abstract}

\section{Introduction}\label{sec:introduction}

A finite semifield is a finite division algebra whose multiplication need not be associative. Its addition is that of a vector space over a finite field, and multiplication is a nonsingular bilinear map over its centre. The subject originates in Dickson's work \cite{Dickson1906commutative} in 1906 on nonassociative finite algebras and has developed through its connections with projective planes, linear groups and finite geometry. Knuth's treatment \cite{Knuth} established isotopy as a central equivalence relation and organized the six operations now associated with the Knuth orbit. A detailed account of these developments and of the construction known around 2010 is given by Lavrauw and Polverino \cite{LP}.

A semifield coordinatizes a translation plane, and isotopic semifields give isomorphic semifield planes. The same multiplication can be studied through its left multiplication maps. A semifield whose multiplication is $K$-bilinear on an $n$-dimensional vector space over $K=\F_q$ yields an $n$-dimensional subspace of $M_n(K)$ in which every nonzero matrix is invertible. In the rank metric, this is a linear maximum rank distance code of minimum distance $n$. Thus constructions of semifields provide extremal rank-metric codes, while equivalence and idealiser methods for matrix codes provide tools for studying isotopy. The underlying rank-metric framework goes back to Delsarte \cite{Delsarte}; see also \cite{SheekeyMRD,SheekeySkew} for more recent constructions and their relation with semifields.

There are many substantially different ways of constructing semifields. Albert's generalized twisted fields \cite{Albert} are obtained by perturbing field multiplication by a second bilinearized monomial. The Hughes--Kleinfeld construction \cite{HK} and the several families introduced by Knuth \cite{Knuth} give further nonassociative multiplications, often presented on a product of two finite fields. Another approach uses skew polynomial rings. Petit's right-remainder construction dates to 1966--1967 \cite{Petit}; over finite fields its relation with cyclic semifields is established in \cite{LSSkew}. Sheekey's construction \cite{SheekeySkew} extends this skew-polynomial approach and includes both cyclic semifields and generalized twisted fields as special cases. Further constructions using skew polynomial quotients were obtained by Lobillo, Santonastaso and Sheekey \cite{LSS}.

Commutative constructions form another large part of the subject. In characteristic two, Knuth's binary presemifields \cite{KnuthBinary} were generalized by Kantor \cite{Kantor03} using a tower of trace maps. Their Knuth orbits also contain the symplectic semifields of Kantor and Williams \cite{KW}. The examples of Ganley \cite{Ganley81} and Cohen--Ganley \cite{CG} are classical instances. In odd characteristic a commutative semifield multiplication gives the planar function $Q(x)=x\circ x$, since
\[
 Q(x+a)-Q(x)=2(a\circ x)+a\circ a
\]
is a permutation for every $a\ne0$. Quadratic planar functions and the related Dembowski--Ostrom polynomials therefore lead to further constructions. These include the Budaghyan--Helleseth \cite{BH08,BH11} and Zha--Kyureghyan--Wang \cite{ZKW} families. Their formulas are closely related to methods used to construct almost perfect nonlinear (APN, for short) functions, although planarity in odd characteristic and the APN property are different conditions.

Bivariate constructions also connect these questions with nonlinear functions. Zhou and Pott \cite{ZP} introduced a two-parameter family of commutative semifields and determined its isotopy classes. Taniguchi's construction \cite{Taniguchi} supplies noncommutative examples related to his quadratic APN functions; G\"olo\u{g}lu and K\"olsch \cite{GKT} subsequently determined when two Taniguchi semifields are isotopic. In a different bivariate construction, G\"olo\u{g}lu and K\"olsch \cite{GK} obtained exponentially many nonisotopic commutative semifields. A further commutative family is obtained from bijections of the Desarguesian plane in \cite{GK2026}. K\"olsch \cite{Kolsch} gives two broader constructions that contain many bivariate families. Such results make clear that the division property, the nuclei, and a complete determination of isotopy are separate aspects of a construction. A nucleus calculation is often an effective first comparison, but families with the same nuclear parameters can require finer invariants.

The prime-dimensional classification problem goes back to Kaplansky's work on three-dimensional division algebras \cite{Kaplansky75,Kaplansky76}. His three-dimensional conjecture was proved by Menichetti \cite{Menichetti77}. Kaplansky also conjectured that a five-dimensional division algebra over $\F_q$ must be a field or a twisted field once $q$ is sufficiently large \cite{Kaplansky75}. This is the conjecture reproduced as (K2) in \cite[p.~69]{Menichetti96}, and in this paper we present a new construction yielding semifields of order $q^5$ with centre of order $q$ for every prime power $q\equiv1\pmod3$, none isotopic to a field or an Albert's generalized twisted field. These examples provide counterexamples to Kaplansky's conjecture over arbitrarily large finite base fields.

In 1996, Menichetti asserted an affirmative answer in the stronger form that, for every prime $r$ and all sufficiently large prime powers $q$, every $r$-dimensional division algebra over $\F_q$ is a field or a generalized twisted field \cite[Corollary~33]{Menichetti96}. His final paragraph explicitly identifies the case $r=5$ with Kaplansky's conjecture. There are gaps in the proof of the higher-dimensional assertion, and our construction gives counterexamples to its conclusion. We discuss the relevant issue in Section~\ref{subsec:menichetti}. The separate three-dimensional classification remains unaffected.

In this paper we construct semifields of order $q^n$ for every prime power $q\equiv1\pmod3$ and every integer $n\ge5$ coprime to $6$. No member of this family is isotopic to a field or an Albert's generalized twisted field. Each semifield in our construction has centre and all three nuclei of order $q$, and the entire Knuth orbit has no commutative isotope.

We determine the isotopism between any two semifields in our family. For fixed $q=p^e\equiv1\pmod3$ and $n$, the family contains exactly $\varphi(n)$ $\F_q$-linear isotopy classes, where $\varphi$ is Euler's totient function. Under isotopy, the number is $\varphi(n)$ when $p\equiv1\pmod3$ and $\varphi(n)/2$ when $p\equiv2\pmod3$; in the latter case $e$ is necessarily even. 

Our nonisotopy proof uses the ranks of multiplication matrices at the vertices determined by the linear factors of a determinant. The determinant defines these vertices, but does not determine the ranks of its representing pencil there. For a finite presemifield, Frobenius acts transitively on these vertices, so all their ranks are equal. We use this common determinant-vertex rank as an isotopy invariant and combine it with the minimal field of definition of the absolute determinant components. To our knowledge, this isotopy invariant has not previously been used in the semifield literature. We also compute its value for finite fields, Albert's generalized twisted fields and Dickson semifields.

The rest part of this paper is organized as follows: Section~\ref{sec:preliminaries} recalls isotopy, nuclei, Ganley's criterion and Dickson matrices. Section~\ref{sec:construction} gives the construction and its uniform determinant factorization. Section~\ref{sec:invariant} defines the invariants and calculates them for several classical constructions of semifields. In Section~\ref{sec:structure}, we determine the left, middle and right nuclei of our new semifields, prove their nonisotopy to finite fields and Albert's generalized twisted fields, classify the isotopisms between members of the family, and address the prime-dimensional classification assertion.

\section{Preliminaries}\label{sec:preliminaries}

\subsection{Semifields, isotopisms and nuclei}

\begin{definition}\label{def:semifield}
A \emph{finite presemifield} is a triple $\Ps=(\mathbb P,+,*)$, where $|\mathbb P|>1$, $(\mathbb P,+)$ is an abelian group, multiplication is distributive over addition on both sides, and
\[
 x*y=0\quad\Longrightarrow\quad x=0\text{ or }y=0.
\]
A \emph{finite semifield} $\Ss=(\mathbb S,+,\circ)$ is a finite presemifield with a multiplicative identity.
\end{definition}

All presemifields in this paper are finite. Their additive groups are elementary abelian, so their orders are prime powers. Multiplication by any nonzero element on either side is an additive bijection. If a field $K$ acts on $(\mathbb P,+)$, we say explicitly that the multiplication is \emph{$K$-bilinear} when
\[
 (\lambda x)*y=\lambda(x*y)=x*(\lambda y)
 \qquad(\lambda\in K,\ x,y\in\mathbb P).
\]
The underlying vector space and its scalar action are fixed whenever a base field is specified.

\begin{definition}\label{def:isotopy}
Let $\Ss=(\mathbb S,+,*)$ and $\Ss'=(\mathbb S',+,\circ)$ be semifields of the same order. An \emph{isotopism} from $\Ss$ to $\Ss'$ is a triple $(A,B,C)$ of additive bijections $\mathbb S\to\mathbb S'$ such that
\begin{equation}\label{eq:isotopy}
 (Ax)\circ(By)=C(x*y)\qquad(x,y\in\mathbb S).
\end{equation}
The semifields are \emph{isotopic} if such a triple exists. The same definition applies to presemifields, without assuming multiplicative identities. If the two multiplications are $K$-bilinear, the isotopism is \emph{$K$-linear} when all three maps are $K$-linear. We write $\Ps\sim\Qs$ for isotopy and $\Ps\sim_K\Qs$ for $K$-linear isotopy.
\end{definition}

Every additive map in characteristic $p$ is $\F_p$-linear. Throughout the paper, isotopism without extra requirement means exactly Definition~\ref{def:isotopy}; a $K$-linearity hypothesis is stated separately when required. The isotopisms from a presemifield $\Ps$ to itself form its \emph{autotopism group}, denoted by $\Aut(\Ps)$. Thus $\Aut(\Ps)$ consists of triples of maps. An isotopism $(U,V,W)$ from $\Ps$ to $\Qs$ induces the group isomorphism
\[
 (A,B,C)\longmapsto(UAU^{-1},VBV^{-1},WCW^{-1}).
\]
In particular, its order depends only on the isotopy class.

We use Kaplansky's trick in the form given in \cite[Proposition~9]{Menichetti96}.
\begin{lemma}[Kaplansky's trick]\label{lem:kaplansky}
Let $\Ps=(\mathbb P,+,*)$ be a presemifield and choose $a\ne0$. Put $L_a(y)=a*y$ and $R_a(x)=x*a$. Then $(\mathbb P,+,\circ)$, with
\[
 x\circ y=R_a^{-1}(x)*L_a^{-1}(y),
\]
is a semifield isotopic to $\Ps$, with identity $a*a$. If $*$ is $K$-bilinear, this isotopism is $K$-linear.
\end{lemma}

For a semifield $\Ss=(\mathbb S,+,\circ)$, its \emph{left nucleus}, \emph{middle nucleus}, \emph{right nucleus} and \emph{centre} are, respectively,
\begin{align*}
 \Nl(\Ss)&=\{a:a\circ(x\circ y)=(a\circ x)\circ y\text{ for all }x,y\},\\
 \Nm(\Ss)&=\{a:x\circ(a\circ y)=(x\circ a)\circ y\text{ for all }x,y\},\\
 \Nr(\Ss)&=\{a:x\circ(y\circ a)=(x\circ y)\circ a\text{ for all }x,y\},\\
 \Z(\Ss)&=\{a\in\Nl(\Ss)\cap\Nm(\Ss)\cap\Nr(\Ss):a\circ x=x\circ a\text{ for all }x\}.
\end{align*}
These are finite fields, and their orders are invariant under isotopy; see \cite{Knuth,LP,MP}. A semifield is a vector space over each nucleus on the appropriate side. If its multiplication is $K$-bilinear, then $K$ embeds in the centre by $\lambda\mapsto\lambda1$. Its dimension over any nucleus consequently divides its dimension over $K$.

\subsection{Spread sets, nuclei and semilinearity}\label{subsec:spreadsets}\label{subsec:semilinearity}

Let $\Ps=(\mathbb P,+,*)$ have order $p^N$, and put $k=\F_p$ and $V=(\mathbb P,+)$. Following \cite{MP}, define its \emph{spread set} by
\begin{equation}\label{eq:spreadset}
 \mathscr C(\Ps)=\{R_y:y\in\mathbb P\}\subseteq\End_k(V),
 \qquad R_y(x)=x*y.
\end{equation}
It is an $N$-dimensional $k$-subspace in which every nonzero map is invertible. We compose maps from right to left. The \emph{dual presemifield} $\Ps^d=(\mathbb P,+,*^d)$ has $x*^dy=y*x$, so
\[
 \mathscr C(\Ps^d)=\{L_x:x\in\mathbb P\},\qquad L_x(y)=x*y.
\]
The distinction between these two spread sets will fix the left and right conventions below.

The spread-set formulation of isotopy is given in \cite[Proposition~2.1]{MPIsotopy}; see also \cite[Proposition~2.1]{MP}.
\begin{proposition}\label{prop:spread-isotopy}
Let $\Ps=(\mathbb P,+,*)$ and $\Qs=(\mathbb Q,+,\circ)$ be presemifields, and let $A,C:\mathbb P\to\mathbb Q$ be additive bijections. There is an isotopism with first and third components $A,C$ if and only if
\begin{equation}\label{eq:spread-equivalence}
 \mathscr C(\Qs)=C\mathscr C(\Ps)A^{-1}.
\end{equation}
In this case the second component $B$ is the unique additive bijection satisfying
\[
 R^{\Qs}_{By}=C R^{\Ps}_y A^{-1}\qquad(y\in\mathbb P).
\]
\end{proposition}

The multiplication of a presemifield has no identity in general. Its nuclear parameters therefore mean those of a semifield obtained by isotopy, as in Lemma~\ref{lem:kaplansky}. They can nevertheless be computed directly from its spread set, without first calculating that semifield multiplication. For a spread set $\mathscr C\subseteq\End_k(V)$, its \emph{left and right idealisers} are
\begin{equation}\label{eq:idealisers}
 I_l(\mathscr C)=\{T\in\End_k(V):T\mathscr C\subseteq\mathscr C\},
 \qquad
 I_r(\mathscr C)=\{T\in\End_k(V):\mathscr C T\subseteq\mathscr C\}.
\end{equation}
For $\omega\in\mathscr C\setminus\{0\}$, also put
\[
 \mathscr Z_\omega(\mathscr C)=
 \{T\in I_l(\mathscr C):TU=U\omega^{-1}T\omega
       \text{ for every }U\in\mathscr C\}.
\]
The following result in idealiser notation can be found in \cite[Theorem~2.2]{MP}.
\begin{proposition}\label{prop:spread-nuclei}
Let $\Ss=(\mathbb S,+,\circ)$ be a semifield isotopic to $\Ps$, and let $\mathscr C=\mathscr C(\Ps)$. The sets below are finite fields under addition and composition, and
\begin{equation}\label{eq:spread-nuclei}
 \begin{aligned}
  \Nl(\Ss)&\cong I_l\bigl(\mathscr C(\Ps^d)\bigr),&
  \Nm(\Ss)&\cong I_r(\mathscr C),\\
  \Nr(\Ss)&\cong I_l(\mathscr C),&
  \Z(\Ss)&\cong\mathscr Z_\omega(\mathscr C).
 \end{aligned}
\end{equation}
The centre isomorphism holds for every fixed $\omega\in\mathscr C\setminus\{0\}$.
\end{proposition}

When multiplication is $K$-bilinear and $\dim_KV=n$, the spread set in \eqref{eq:spreadset} is an $n$-dimensional $K$-subspace of $\End_K(V)$. In matrices, it is a \emph{maximum rank distance code}: every nonzero matrix has rank $n$, and its size $|K|^n$ attains the rank-metric Singleton bound \cite{Delsarte}. Indeed, deleting any $n-1$ rows is injective on any code of minimum rank distance $n$, so that bound is $|K|^n$. Conversely, an $n$-dimensional $K$-linear matrix space of this kind defines a presemifield with $K$-bilinear multiplication after its elements are labelled $K$-linearly by $V$. The graphs of its maps, together with the vertical subspace, form a semifield spread. The same observations apply to $\mathscr C(\Ps^d)$.

The restriction on isotopisms that we need is the following result \cite[Theorem~2.2]{MPIsotopy}, restated in \cite[Theorem~2.9]{MP}.
\begin{lemma}[Semilinearity]\label{lem:semilinearity}
Let $\Ps$ and $\Qs$ be presemifields of order $p^N$, with their additive groups identified with $\F_{p^N}$. Suppose their spread sets are contained in $\End_K(\F_{p^N})$, where $K$ is a subfield of $\F_{p^N}$. If $(A,B,C)$ is an isotopism from $\Ps$ to $\Qs$, then $A$ and $C$ are $K$-semilinear maps with the same companion automorphism of $K$.
\end{lemma}

If both multiplications are $K$-bilinear, their dual spread sets are also contained in $\End_K(\F_{p^N})$. Applying the lemma to $(B,A,C)$ between the dual presemifields shows that all three maps have one common companion automorphism $\theta\in\Gal(K/\F_p)$. Choose a $\theta$-semilinear coordinate bijection $\phi_\theta$ and define the \emph{Galois twist} $\Ps^{[\theta]}=(\mathbb P,+,*^{[\theta]})$ by
\begin{equation}\label{eq:galois-twist}
 x*^{[\theta]}y
 =\phi_\theta\bigl(\phi_\theta^{-1}(x)*\phi_\theta^{-1}(y)\bigr).
\end{equation}
Factoring this common semilinear map from the isotopism gives
\begin{equation}\label{eq:semilinear-reduction}
 \Ps\sim\Qs
 \quad\Longleftrightarrow\quad
 \Ps^{[\theta]}\sim_K\Qs
 \text{ for some }\theta\in\Gal(K/\F_p).
\end{equation}
The field-automorphism twist cannot generally be discarded for two fixed presemifields. For the present family it will change only the coefficient parameter.

We also use Ganley's commutativity criterion \cite[Theorem~4]{Ganley72} and its presemifield extension due to Bierbrauer \cite[Section~6]{Bierbrauer2016}.
\begin{lemma}[Ganley--Bierbrauer]\label{lem:ganley}
A presemifield $\Ps=(\mathbb P,+,*)$ is isotopic to a commutative semifield if and only if there is an additive bijection $H$ of $\mathbb P$ such that
\begin{equation}\label{eq:ganley-pre}
 x*(Hy)=y*(Hx)\qquad(x,y\in\mathbb P).
\end{equation}
For a semifield $\Ss=(\mathbb S,+,\circ)$, this is equivalent to the existence of $a\ne0$ satisfying
\begin{equation}\label{eq:ganley-semifield}
 (a\circ x)\circ y=(a\circ y)\circ x
 \qquad(x,y\in\mathbb S).
\end{equation}
\end{lemma}

To relate the stated presemifield form to Bierbrauer's notation, fix $u\ne0$ and let $\alpha=R_u^{-1}$. His criterion is the existence of $v\ne0$ such that $Ux*y=Uy*x$, where $U=\alpha L_v$. Substituting $x=H x'$, $y=H y'$ with $H=U^{-1}$ gives \eqref{eq:ganley-pre}. Thus the displayed form changes only the placement of the invertible map, not the criterion.

\subsection{Linearized polynomials and Dickson matrices}

Let $K=\F_h$ and $F=\F_{h^n}$. Every $K$-linear endomorphism of $F$ has a unique expression
\begin{equation}\label{eq:linearized}
 L(x)=\sum_{i=0}^{n-1}a_i x^{h^i},\qquad a_i\in F.
\end{equation}
Such an expression is a \emph{linearized polynomial}. Uniqueness follows from the root bound for a nonzero polynomial of degree at most $h^{n-1}$; existence then follows by comparing the $K$-dimensions of the coefficient space and $\End_K(F)$.

Its \emph{Dickson matrix}, denoted by $D_L$, is
\begin{equation}\label{eq:dickson}
 D_L=(a_{j-i}^{h^i})_{0\le i,j<n},
\end{equation}
with subscripts modulo $n$. If
\[
 \Phi_h(x)=(x,x^h,\ldots,x^{h^{n-1}})^{\mathsf T},
\]
then $D_L\Phi_h(x)=\Phi_h(L(x))$. We recall the standard rank correspondence; see \cite[Lemma~3 and Proposition~6]{Menichetti96} and \cite{WuLiu}.
\begin{lemma}\label{lem:moore}
Let $b_0,\ldots,b_{n-1}$ be a $K$-basis of $F$ and put $M=(b_j^{h^i})_{i,j}$. Then $M$ is invertible. If $P_L$ is the matrix of $L$ in this basis, then
\[
 D_LM=MP_L,\qquad \rank_F D_L=\rank_K P_L.
\]
In particular, $L$ is invertible if and only if $\det D_L\ne0$.
\end{lemma}

Let $\Ps=(\mathbb P,+,*)$ have $K$-bilinear multiplication and dimension $n$ over $K$. After fixing a $K$-linear identification of $(\mathbb P,+)$ with $(F,+)$, write its multiplication polynomial as
\begin{equation}\label{eq:bilinearized}
 f_*(x,y):=x*y=\sum_{i,j=0}^{n-1}c_{ij}x^{h^i}y^{h^j}.
\end{equation}
In independent variables $X_i,Y_j$, its \emph{split output forms} are
\[
 f_{*,r}(\Bf X,\Bf Y)=\sum_{i,j}c_{ij}^{h^r}X_{i+r}Y_{j+r},
 \qquad 0\le r<n.
\]
Define the \emph{split multiplication matrices} $L_{\Ps}(\Bf X)$ and $R_{\Ps}(\Bf Y)$ by
\begin{equation}\label{eq:splitmaps}
 L_{\Ps}(\Bf X)\Bf Y=R_{\Ps}(\Bf Y)\Bf X
 =(f_{*,0},\ldots,f_{*,n-1})^{\mathsf T}.
\end{equation}
For field elements $x$ and $y$,
\begin{equation}\label{eq:actualpoints}
 L_{\Ps}(\Phi_h(x))\Phi_h(y)=\Phi_h(x*y).
\end{equation}
The independent variables describe the multiplication after extension of scalars. They are not arbitrary coordinates of an original field element: when $x\ne0$, every entry of $\Phi_h(x)$ is nonzero.

\subsection{Multiplication determinants and the trace form}\label{subsec:determinants}

Fix an ordinary $K$-basis $\mathcal B=(b_0,\ldots,b_{n-1})$ of the additive space of $\Ps$, identified $K$-linearly with $F=\F_{h^n}$. We use $\Bf U,\Bf V,\Bf W$ for ordinary coordinate vectors and $\Bf X,\Bf Y,\Bf Z$ for split coordinate vectors. If $[T]_{\mathcal B}$ is the ordinary matrix of a $K$-linear map $T$, put
\[
 A_{L,\Ps}(\Bf U)=\sum_{j=0}^{n-1}U_j[L_{b_j}]_{\mathcal B},
 \qquad
 A_{R,\Ps}(\Bf V)=\sum_{j=0}^{n-1}V_j[R_{b_j}]_{\mathcal B}.
\]
These are the \emph{ordinary multiplication matrices}. Their determinants are the \emph{left determinant form} and the \emph{right determinant form}, respectively:
\[
 D_{L,\Ps}(\Bf U)=\det A_{L,\Ps}(\Bf U),
 \qquad
 D_{R,\Ps}(\Bf V)=\det A_{R,\Ps}(\Bf V).
\]
They belong to polynomial rings over $K$, have degree $n$, and have no nonzero $K$-zero. We retain $L_{\Ps}(\Bf X)$ and $R_{\Ps}(\Bf Y)$ for the split matrices in \eqref{eq:splitmaps}, and write
\[
 \widetilde D_{L,\Ps}(\Bf X)=\det L_{\Ps}(\Bf X),
 \qquad
 \widetilde D_{R,\Ps}(\Bf Y)=\det R_{\Ps}(\Bf Y)
\]
for their \emph{split determinant forms}.

A third matrix family is defined by the \emph{trace trilinear form}
\begin{equation}\label{eq:tracetensor}
 T_{\Ps}(x,y,z)=\Tr_{F/K}\bigl(z(x*y)\bigr).
\end{equation}
For a trilinear form $T(x,y,z)$ and a fixed vector $x_0$, the map $(y,z)\mapsto T(x_0,y,z)$ is a bilinear form in $y$ and $z$. This operation is called \emph{contraction} of $T$ with $x_0$ in the first variable; contractions in the second and third variables are defined similarly. For \eqref{eq:tracetensor}, the three choices are
\[
 \begin{array}{c|c}
 \text{fixed variable}&\text{bilinear form}\\\hline
 x&(y,z)\mapsto\Tr_{F/K}\bigl(z(x*y)\bigr)\\
 y&(x,z)\mapsto\Tr_{F/K}\bigl(z(x*y)\bigr)\\
 z&(x,y)\mapsto\Tr_{F/K}\bigl(z(x*y)\bigr).
 \end{array}
\]
The first two are multiplication matrices after identifying the output with its dual by the trace pairing. For the third, define the \emph{ordinary trace matrix} by
\[
 \bigl(A_{\mathrm{tr},\Ps}(\Bf W)\bigr)_{ij}
 =\sum_{k=0}^{n-1}W_k\Tr_{F/K}\bigl(b_k(b_i*b_j)\bigr).
\]
Thus $\Bf U^{\mathsf T}A_{\mathrm{tr},\Ps}(\Bf W)\Bf V$ is the ordinary-coordinate expression of the scalar extension of \eqref{eq:tracetensor}. Its determinant
\[
 D_{\mathrm{tr},\Ps}(\Bf W)=\det A_{\mathrm{tr},\Ps}(\Bf W)\in K[\Bf W]
\]
is the \emph{trace determinant form}.

In split coordinates the trace pairing is the dot product, since $\Tr(uv)=\sum_i u^{h^i}v^{h^i}$. Write
\[
 \widetilde T_{\Ps}(\Bf X,\Bf Y,\Bf Z)
 =\sum_{r=0}^{n-1}Z_rf_{*,r}(\Bf X,\Bf Y)
\]
for the \emph{split trace tensor}, and let $M_{\Ps}^{\mathrm{tr}}(\Bf Z)$ be its coefficient matrix in $\Bf X,\Bf Y$. The three expressions in these coordinates are
\begin{equation}\label{eq:three-matrices}
 \begin{aligned}
 \widetilde T_{\Ps}(\Bf X,\Bf Y,\Bf Z)
 &=\Bf Z^{\mathsf T}L_{\Ps}(\Bf X)\Bf Y
 =\Bf Z^{\mathsf T}R_{\Ps}(\Bf Y)\Bf X\\
 &=\Bf X^{\mathsf T}M_{\Ps}^{\mathrm{tr}}(\Bf Z)\Bf Y.
 \end{aligned}
\end{equation}
We set
\[
 \widetilde D_{\mathrm{tr},\Ps}(\Bf Z)
 =\det M_{\Ps}^{\mathrm{tr}}(\Bf Z).
\]
The ordinary trace determinant also has no nonzero $K$-zero. For $z\ne0$ and $x\ne0$, surjectivity of $y\mapsto x*y$ and nondegeneracy of the trace pairing imply that $T_{\Ps}(x,y,z)$ does not vanish for every $y$. Hence the fixed-$z$ bilinear form is nonsingular. Permuting the three variables gives the six Knuth operations, with the usual trace identifications \cite{Knuth,LavrauwTensors}.

Let $M=(b_j^{h^i})_{i,j}$ be the Moore matrix of $\mathcal B$. We relate the ordinary and split polynomial coordinates by the invertible linear substitutions
\[
\Bf X=M\Bf U,\qquad
\Bf Y=M\Bf V,\qquad
\Bf Z=M\Bf W.
\]
For example,
\[
X_i=\sum_{j=0}^{n-1}b_j^{h^i}U_j.
\]
Here the $U_j$ are independent indeterminates, and, since $M$ is invertible, the $X_i$ form another algebraically independent coordinate system. No relations $X_{i+1}=X_i^h$ are imposed. After specializing $U_j=u_j\in K$, however, and putting $x=\sum_jb_ju_j$, we obtain
\[
X_i(\Bf u)=x^{h^i}.
\]
Thus the Frobenius relations hold for the coordinates of actual elements of $F$, but not as identities between the independent polynomial coordinates. The Moore-matrix similarity for the multiplication maps and the change of coordinates in the trace bilinear form give
\begin{equation}\label{eq:ordinary-split-determinants}
 \begin{aligned}
 D_{L,\Ps}(\Bf U)&=\widetilde D_{L,\Ps}(M\Bf U),\\
 D_{R,\Ps}(\Bf V)&=\widetilde D_{R,\Ps}(M\Bf V),\\
 D_{\mathrm{tr},\Ps}(\Bf W)
 &=(\det M)^2\widetilde D_{\mathrm{tr},\Ps}(M\Bf W).
 \end{aligned}
\end{equation}
These are polynomial identities. The scalar $(\det M)^2$ lies in $K^\times$, since $M^{\mathsf T}M$ is the nonsingular Gram matrix
$(\Tr_{F/K}(b_i b_j))_{i,j}$ of the trace pairing.

Norm expressions in ordinary coordinates are understood as
\emph{polynomial norms}:
\begin{equation}\label{eq:polynomial-norm}
 \N_{F/K}\!\left(\sum_j b_jU_j\right)
 =\prod_{i=0}^{n-1}\left(\sum_j b_j^{h^i}U_j\right).
\end{equation}
The conjugations act on coefficients and fix the independent variables
$U_j$, rather than raising those variables to powers.

Let $(A,B,C)$ be a $K$-linear isotopism from $\Ps$ to $\Qs$. Denote its ordinary coordinate matrices by $A_0,B_0,C_0$, and its split matrices by $A_{\mathrm{sp}},B_{\mathrm{sp}},C_{\mathrm{sp}}$. In ordinary coordinates,
\[
 A_{L,\Qs}(A_0\Bf U)B_0=C_0A_{L,\Ps}(\Bf U),
 \qquad
 A_{R,\Qs}(B_0\Bf V)A_0=C_0A_{R,\Ps}(\Bf V).
\]
In the split coordinates the left identity is
\begin{equation}\label{eq:pencilisotopy}
 L_{\Qs}(A_{\mathrm{sp}}\Bf X)B_{\mathrm{sp}}
 =C_{\mathrm{sp}}L_{\Ps}(\Bf X),
\end{equation}
and the corresponding split trace tensors satisfy
\[
 \widetilde T_{\Qs}(A_{\mathrm{sp}}\Bf X,B_{\mathrm{sp}}\Bf Y,
 C_{\mathrm{sp}}^{-\mathsf T}\Bf Z)
 =\widetilde T_{\Ps}(\Bf X,\Bf Y,\Bf Z).
\]
For clarity, the trace-variable change in ordinary coordinates is generally not $C_0^{-\mathsf T}$. If $G_{\Ps}$ and $G_{\Qs}$ are the Gram matrices of the two trace pairings, set
\[
 C_\vee=G_{\Qs}^{-1}C_0^{-\mathsf T}G_{\Ps}.
\]
Then
\[
 A_0^{\mathsf T}A_{\mathrm{tr},\Qs}(C_\vee\Bf W)B_0
 =A_{\mathrm{tr},\Ps}(\Bf W).
\]
The matrix $C_\vee$ represents the inverse of the trace adjoint of $C$ and is invertible over $K$. Each ordinary determinant form is therefore preserved up to an invertible $K$-linear substitution and a nonzero scalar. The split forms satisfy the corresponding identities in their specified coordinates. These are the matrix forms of Proposition~\ref{prop:spread-isotopy}, using the left spread set when a left multiplication matrix occurs.

\section{Construction}\label{sec:construction}

Throughout this section, let
\begin{equation}\label{eq:admissible}
 q=p^e\equiv1\pmod3,\qquad n\ge5,\qquad \gcd(n,6)=1,
 \qquad K=\F_q,\quad F=\F_{q^n}.
\end{equation}
Here $p$ is prime and $e\ge1$; the extension degree $n=[F:K]$ is independent of $e$. The condition on $q$ allows every $e\ge1$ if $p\equiv1\pmod3$, and exactly the even $e$ if $p\equiv2\pmod3$. Thus $|F|=p^{en}$, while the dimension over $K$ is $n$.

Choose $w\in K$ satisfying
\begin{equation}\label{eq:parameter}
 w^2-w+1=0.
\end{equation}
Such an element exists because $K^\times$ has an element $\omega$ of order three, and one can take $w=-\omega$. The characteristic is not three, and
\begin{equation}\label{eq:w-relations}
 w^3=-1,\qquad w^{-1}=1-w,\qquad w\ne0,-1,\qquad w\ne w^{-1}.
\end{equation}
Let $s$ be a unit modulo $n$ and put $\sigma(x)=x^{q^s}$. Thus $\sigma$ generates $\Gal(F/K)$. Negative superscripts on $\sigma$ denote inverse automorphisms, not reciprocals of field elements.

Define $\Ps_{w,s}=(F,+,*_{w,s})$ by
\begin{equation}\label{eq:construction}
\begin{split}
 x*_{w,s}y={}&x^\sigma y^{\sigma^2}
 +x^\sigma y^{\sigma^{-1}}
 +x^{\sigma^{-2}}y^{\sigma^{-1}}\\
 &+w\bigl(x^{\sigma^2}y^\sigma
 +x^{\sigma^{-1}}y^\sigma
 +x^{\sigma^{-1}}y^{\sigma^{-2}}\bigr).
\end{split}
\end{equation}
In the case $n=5$, $s=1$, the exponents in the first line are $(1,2),(1,4),(3,4)$ in $q$-Frobenius notation.

For polynomial calculations, write $f_{w,s}(x,y)=x*_{w,s}y$. Use independent coordinates $X_i,Y_i$ indexed by $i\in\mathbb Z/n\mathbb Z$ in the $\sigma$-ordering of the embeddings. The \emph{split output forms} are
\begin{equation}\label{eq:family-split}
\begin{split}
 f_i={}&X_{i+1}Y_{i+2}+X_{i+1}Y_{i-1}+X_{i-2}Y_{i-1}\\
 &+w\bigl(X_{i+2}Y_{i+1}+X_{i-1}Y_{i+1}+X_{i-1}Y_{i-2}\bigr).
\end{split}
\end{equation}
Write $L_w(\Bf X)=L_{\Ps_{w,s}}(\Bf X)$ in this $\sigma$-ordering; it is the coefficient matrix in $\Bf Y$. Its only possibly nonzero entries in row $i$ are
\begin{equation}\label{eq:family-entries}
\begin{aligned}
 (L_w)_{i,i+2}&=X_{i+1},&
 (L_w)_{i,i+1}&=w(X_{i+2}+X_{i-1}),\\
 (L_w)_{i,i-1}&=X_{i+1}+X_{i-2},&
 (L_w)_{i,i-2}&=wX_{i-1}.
\end{aligned}
\end{equation}
The four positions are distinct since $n\ge5$.

\begin{theorem}\label{thm:division}
For all parameters in \eqref{eq:admissible}, $\Ps_{w,s}$ is a presemifield. Its split left determinant is
\begin{equation}\label{eq:leftdet}
 \widetilde D_{L,\Ps_{w,s}}(\Bf X)=\det L_w(\Bf X)
 =\lambda_n(w)\prod_{i=0}^{n-1}X_i,
 \qquad \lambda_n(w)=(1+w^n)^2(1-w^{2n})\ne0.
\end{equation}
More explicitly,
\begin{equation}\label{eq:lambda}
 \lambda_n(w)=
 \begin{cases}
 3(1+w),&n\equiv1\pmod6,\\
 3(2-w),&n\equiv5\pmod6.
 \end{cases}
\end{equation}
Its ordinary-coordinate left determinant $D_{L,\Ps_{w,s}}$ is $\lambda_n(w)\N_{F/K}$, with the norm understood as in \eqref{eq:polynomial-norm}.
\end{theorem}
\begin{proof}
We prove a matrix factorization that is uniform in $n$. Write $\Pi$ for the cyclic permutation matrix with $(\Pi u)_i=u_{i+1}$. If $T=\diag(t_i)\Pi^a$ with $\gcd(a,n)=1$, then
\begin{equation}\label{eq:weighted-cycle}
 \det(zI-T)=z^n-\prod_i t_i.
\end{equation}
Indeed, choosing a shifted entry in the determinant expansion forces the same choice around the full $n$-cycle. Only the identity and the full cycle contribute. This argument is valid even when the characteristic divides $n$.

In $K(X_0,\ldots,X_{n-1})$, set
\[
 \Delta=\diag(X_i),\qquad H=\diag(X_{i-1}X_{i+1}),\qquad
 B=\diag\left(\frac{X_{i+2}}{X_{i-1}}\right)\Pi.
\]
We claim that
\begin{equation}\label{eq:factorization}
 L_w(\Bf X)\Delta
 =H(B+wI)(I-w^2B^{-1}\Pi^{-1})(\Pi+wI).
\end{equation}
The inverse weighted shift satisfies
\[
 (B^{-1}u)_i=\frac{X_{i-2}}{X_{i+1}}u_{i-1}.
\]
Using $w^3=-1$, expansion of the last three factors in \eqref{eq:factorization} gives
\begin{align*}
 &(B+wI)(I-w^2B^{-1}\Pi^{-1})(\Pi+wI)\\
 &\hspace{8mm}=B\Pi+wB+w\Pi+\Pi^{-1}+B^{-1}+wB^{-1}\Pi^{-1}.
\end{align*}
The entries of this matrix in columns $i+2,i+1,i-1,i-2$ of row $i$ are
\[
 \frac{X_{i+2}}{X_{i-1}},\quad
 w\left(\frac{X_{i+2}}{X_{i-1}}+1\right),\quad
 1+\frac{X_{i-2}}{X_{i+1}},\quad
 w\frac{X_{i-2}}{X_{i+1}},
\]
respectively. Multiplication by $H$ gives exactly the corresponding entries of $L_w\Delta$ in \eqref{eq:family-entries}, proving the claim.

The product of the weights of $B$ is one, and $n$ is odd. Thus \eqref{eq:weighted-cycle} gives
\[
 \det(B+wI)=1+w^n=\det(\Pi+wI).
\]
Also
\[
 (B^{-1}\Pi^{-1}u)_i=\frac{X_{i-2}}{X_{i+1}}u_{i-2}.
\]
The shift by $-2$ is an $n$-cycle, and its weights have product one. Consequently,
\[
 \det(I-w^2B^{-1}\Pi^{-1})=1-w^{2n}.
\]
Since $\det\Delta=\prod_iX_i$ and $\det H=(\prod_iX_i)^2$, taking determinants in \eqref{eq:factorization} proves \eqref{eq:leftdet}. It is a polynomial identity, although it was obtained in a rational-function field. Equivalently, all the calculations take place in the Laurent-polynomial ring over $\mathbb Z[w]/(w^2-w+1)$; no restrictions on specialization arise from the temporary denominators.

As $n$ is coprime to $6$, $w^n$ is $w$ or $w^{-1}$. Using \eqref{eq:w-relations} gives \eqref{eq:lambda}. The scalar is nonzero: either $w=-1$ or $w=2$, together with \eqref{eq:parameter}, would force characteristic three.

For an actual element $x\in F$, put
\[
 \Phi_\sigma(x)=(x,x^\sigma,\ldots,x^{\sigma^{n-1}})^{\mathsf T}.
\]
The defining equations give
\begin{equation}\label{eq:actual-family}
 L_w(\Phi_\sigma(x))\Phi_\sigma(y)=\Phi_\sigma(x*_{w,s}y).
\end{equation}
For $x\ne0$, the determinant on the left is
\[
 \lambda_n(w)\prod_{i=0}^{n-1}x^{\sigma^i}
 =\lambda_n(w)\N_{F/K}(x)\ne0.
\]
Therefore $x*_{w,s}y=0$ implies $y=0$. The multiplication is $K$-bilinear and has no zero divisors, so $\Ps_{w,s}$ is a presemifield. The $\sigma$-Moore matrix is obtained from the usual Moore matrix by permuting its rows. Using \eqref{eq:ordinary-split-determinants} in this ordering identifies $\prod_iX_i$ with the ordinary norm form \eqref{eq:polynomial-norm}, proving the final assertion.
\end{proof}

The distinction between split and ordinary coordinates is essential. A nonzero actual tuple $\Phi_\sigma(x)$ never has a zero coordinate. A vector such as $e_0=(1,0,\ldots,0)$ is a legitimate split parameter, but does not represent a nonzero element of $F$ in conjugate coordinates. Ordinary field multiplication has the same distinction: its split pencil is diagonal, although its original algebra has no zero divisors. The ranks at such split points will be used for isotopy later in Section \ref{sec:invariant}.

\begin{remark}\label{rem:division}
	The determinant formula in the proof of Theorem \ref{thm:division} holds for every odd $n\ge5$. If $3\mid n$, its scalar is zero, so that case is excluded by this construction. The argument does not cover even $n$. No condition $p\nmid n$ was used; in particular examples with $p=n$ are permitted whenever $p\equiv1\pmod3$.
\end{remark}

\begin{lemma}\label{lem:cyclic}
The split right and trace determinants $\widetilde D_{R,\Ps_{w,s}}$ and $\widetilde D_{\mathrm{tr},\Ps_{w,s}}$ are also $\lambda_n(w)$ times their respective coordinate products. The ordinary right and trace forms are nonzero scalar multiples of the corresponding norm form, by \eqref{eq:ordinary-split-determinants} and \eqref{eq:polynomial-norm}. The split trace tensor $T_w:=\widetilde T_{\Ps_{w,s}}$ satisfies
\begin{equation}\label{eq:cyclic-tensor}
 T_w(X,Y,Z)=T_w(Y,Z,X)=T_w(Z,X,Y).
\end{equation}
Moreover,
\begin{equation}\label{eq:opposite}
 f_{w,s}(y,x)=w f_{w^{-1},s}(x,y),\qquad
 f_{w,-s}(x,y)=w f_{w^{-1},s}(x,y).
\end{equation}
\end{lemma}
\begin{proof}
The first identity in \eqref{eq:opposite} follows by interchanging the two groups of three terms in \eqref{eq:construction}; the second follows by replacing $\sigma$ by $\sigma^{-1}$. Thus the split right determinant is
\[
 \widetilde D_{R,\Ps_{w,s}}(\Bf Y)=w^n\lambda_n(w^{-1})\prod_iY_i=\lambda_n(w)\prod_iY_i,
\]
where the last equality uses $w^{3n}=-1$. In $T_w=\sum_iZ_if_i$, the coefficient-one terms are the three cyclic orientations of the consecutive triple $(i,i+1,i+2)$, and the coefficient-$w$ terms have the reverse orientation. Reindexing proves \eqref{eq:cyclic-tensor}. In particular,
\begin{equation}\label{eq:trace-left}
 M^{\mathrm{tr}}_{\Ps_{w,s}}(Z)=L_w(Z)^{\mathsf T},
\end{equation}
which proves the remaining determinant identity.
\end{proof}

\begin{corollary}\label{cor:knuth}
Every member of the Knuth orbit of $\Ps_{w,s}$ is isotopic to either $\Ps_{w,s}$ or $\Ps_{w^{-1},s}$. Equivalently, the two representatives can be taken as $\Ps_{w,s}$ and $\Ps_{w,-s}$.
\end{corollary}
\begin{proof}
The tensor is unchanged by the cyclic permutations of its three variables, and a transposition gives the opposite multiplication. Apply \eqref{eq:opposite}. The same statements hold for the absolute trace after composition with $\Tr_{K/\F_p}$, so they hold for isotopy.
\end{proof}

\section{Splitting Fields and Determinant-Vertex Ranks}\label{sec:invariant}

We introduce invariants of presemifields under isotopy and compute them for several classical families. The three determinant forms record the possible singular multiplication maps after extension of scalars. For an intrinsic comparison we use the full centre; auxiliary calculations over smaller central subfields are related by restriction of scalars. Their irreducible components determine fields of definition and, when those components are hyperplanes, a distinguished set of projective points at which matrix ranks can be evaluated. The fields and component degrees describe the factorization, while the vertex rank records information about its matrix representation. In Section~\ref{sec:structure} these invariants will distinguish our construction from fields and Albert's generalized twisted fields, including cases in which the nuclear orders coincide.

\subsection{Component splitting fields and determinant vertices}


We first work with $\Ps=(\mathbb P,+,*)$ whose multiplication is
$K$-bilinear, where $K=\F_h$ and $\dim_K\mathbb P=n$. The following
lemma transfers the split-coordinate calculations to the ordinary
$K$-coordinates in which the invariants are defined.

\begin{lemma}\label{lem:coordinate-transfer}
Identify the additive space of $\Ps$ with $F=\F_{h^n}$, choose an
ordinary $K$-basis $\mathcal B$, and let $M$ be its Moore matrix.
Fix $\mu\in\{L,R,\mathrm{tr}\}$. Over any field containing $F$,
substitution by $M$ gives a correspondence between the factorizations
of $\widetilde D_{\mu,\Ps}$ and $D_{\mu,\Ps}$, preserving
irreducibility, factor degrees and multiplicities. More precisely,
\begin{equation}\label{eq:factor-coordinate-transfer}
 \begin{aligned}
 \widetilde D_{\mu,\Ps}(\Bf X)
 &=a\prod_i Q_i(\Bf X)^{e_i}
 \quad\Longrightarrow\quad
 D_{\mu,\Ps}(\Bf U)
 =\varepsilon_\mu a\prod_i Q_i(M\Bf U)^{e_i},\\
 &\hspace{12mm}\varepsilon_L=\varepsilon_R=1,
 \qquad \varepsilon_{\mathrm{tr}}=(\det M)^2,
 \end{aligned}
\end{equation}
where $\Bf X$ and $\Bf U$ denote the split and ordinary parameter
vectors for the chosen determinant type.

Let $\tau$ act on coefficients in $\overline K$ by $a\mapsto a^h$
and fix the ordinary variables, and put $(\Pi\Bf X)_i=X_{i+1}$,
with indices modulo $n$. For $Q\in\overline K[\Bf X]$,
\begin{equation}\label{eq:split-factor-frobenius}
 \tau\bigl(Q(M\Bf U)\bigr)
 =Q^{(h)}(\Pi M\Bf U),
\end{equation}
where $Q^{(h)}$ is obtained by applying $a\mapsto a^h$ to the
coefficients and fixing the split variables.
\end{lemma}
\begin{proof}
The matrix identities underlying
\eqref{eq:ordinary-split-determinants} are
\begin{equation}\label{eq:ordinary-split-matrices}
 \begin{aligned}
 L_{\Ps}(M\Bf U)&=M A_{L,\Ps}(\Bf U)M^{-1},\\
 R_{\Ps}(M\Bf V)&=M A_{R,\Ps}(\Bf V)M^{-1},\\
 A_{\mathrm{tr},\Ps}(\Bf W)
 &=M^{\mathsf T}M_{\Ps}^{\mathrm{tr}}(M\Bf W)M.
 \end{aligned}
\end{equation}
The first two hold at each ordinary coordinate basis vector by
Lemma~\ref{lem:moore}, and hence identically by linearity in the
parameters. The third follows by substituting
$\Bf X=M\Bf U$, $\Bf Y=M\Bf V$ and $\Bf Z=M\Bf W$ into
\eqref{eq:three-matrices}.

Since $M$ is invertible, substitution by $M$ is an isomorphism of
polynomial rings over any field containing $F$, with inverse given
by substitution by $M^{-1}$. Together with
\eqref{eq:ordinary-split-determinants}, this proves
\eqref{eq:factor-coordinate-transfer} and the assertions about
factorization. Finally, $\tau(M)=\Pi M$, which gives
\eqref{eq:split-factor-frobenius}.
\end{proof}

We now define the relative invariants; choosing the full centre of an isotopic semifield will give invariants without a prescribed scalar field.

\begin{definition}\label{def:splittingfield}
For a nonzero homogeneous polynomial $D\in K[U_0,\ldots,U_{n-1}]$, its \emph{component splitting field} $E_K(D)$ is the smallest extension of $K$ in a fixed algebraic closure $\overline K$ over which every absolutely irreducible factor is defined up to a nonzero scalar. Equivalently, normalize a nonzero coefficient of each factor to one and adjoin its coefficients. For presemifield $\Ps$, put
\[
 E_{\mu,K}(\Ps)=E_K(D_{\mu,\Ps}),\qquad
 \mu\in\{L,R,\mathrm{tr}\}.
\]
The labels $L,R,\mathrm{tr}$ refer to the determinant types defined in Section~\ref{subsec:determinants}.
\end{definition}

The component splitting fields are finite extensions. They refer to the ordinary determinant form $D_{\mu,\Ps}$ and its original $K$-structure; Lemma~\ref{lem:coordinate-transfer} permits their calculation from split factors.

\begin{proposition}\label{prop:splittingfield}
	Let $\Ps=(\mathbb P,+,*)$ be a finite presemifield whose multiplication is $K$-bilinear, where	$K=\F_h$ and $\dim_K\mathbb P=n$.
	The three component splitting fields $E_{\mu,K}(\Ps)$, $\mu\in\{L,R,\mathrm{tr}\}$,	are preserved by $K$-linear isotopy.
	
	Fix $\mu\in\{L,R,\mathrm{tr}\}$ and write
	\[
	D=D_{\mu,\Ps}\in K[U_0,\ldots,U_{n-1}]
	\]
	for the corresponding determinant form in ordinary	$K$-coordinates. Let $m$ be the number of its distinct	absolutely irreducible factors. Each factor has multiplicity one, all have a common degree $d$, and
	\begin{equation}\label{eq:componentdegrees}
		md=n,\qquad
		E_{\mu,K}(\Ps)=E_K(D)=\F_{h^m}.
	\end{equation}
	If $d=1$, the linear factors are independent and $E_{\mu,K}(\Ps)=\F_{h^n}$.
\end{proposition}
\begin{proof}
The ordinary-coordinate identities in Section~\ref{subsec:determinants} show that a $K$-linear isotopy changes each $D_{\mu,\Ps}$ by a nonzero scalar and an invertible $K$-linear substitution. A factor is defined over an extension $E/K$ before the substitution if and only if its image is defined over $E$, since the inverse substitution is also defined over $K$.

For the fixed determinant form $D=D_{\mu,\Ps}$, the matrix size and the number of parameter variables are both $n$. Hence $D$ is homogeneous of degree $n$, and the division property gives
\[
	D(\Bf a)\ne0 \qquad\text{for every }\Bf a\in K^n\setminus\{0\}.
\]
We first show that these properties force $D$ to be irreducible over $K$.  Indeed, suppose that $D=AB$ with $A,B$ nonconstant polynomials over $K$. Since $D$ is homogeneous, both factors are homogeneous, and $1\le\deg A<n$. By the Chevalley--Warning theorem
\cite{LN}, the number of zeros of $A$ in $K^n$ is divisible by $\operatorname{char}K$. The zero vector is a zero because $A$ is homogeneous of positive degree, so there must be another zero. This would be a nonzero zero of $D$, a contradiction.

Write the absolute factorization as
\[
D=c\prod_{i=1}^{r}F_i^{e_i},
\qquad e_i\ge1,
\]
where the $F_i$ are pairwise nonassociate absolutely irreducible polynomials. Fix a monomial order and
normalize each $F_i$ to have leading coefficient one. Then $c\in K^\times$. Let $\tau$ act on coefficients by $a\mapsto a^h$ and fix the variables. Since $\tau(D)=D$, unique factorization shows that $\tau$ permutes the normalized factors $F_i$.

Choose one orbit $\mathcal O$ and set
\[
H=\prod_{F_i\in\mathcal O}F_i,
\]
taking each factor in this orbit once. The polynomial $H$ is fixed by $\tau$, so it belongs to $K[U_0,\ldots,U_{n-1}]$. It divides $D$ over $\overline K$. Its quotient also belongs to $K[U_0,\ldots,U_{n-1}]$: if $D=HQ$, then
\[
H\tau(Q)=\tau(D)=D=HQ,
\]
and hence $\tau(Q)=Q$. Thus $H$ is a nonconstant divisor of $D$ over $K$. The irreducibility of $D$ forces $Q$ to be constant. Consequently, all the absolute factors lie in $\mathcal O$, and every
multiplicity $e_i$ equals one.

Let $m$ be the number of these factors. Frobenius acts on them as an $m$-cycle, and they all have the same degree $d$. Hence $md=n$.
For any normalized factor $F_i$, its coefficients belong to $\F_{h^a}$ if and only if $\tau^a(F_i)=F_i$, which holds precisely when
$m\mid a$. Its minimal field of definition is therefore $\F_{h^m}$. The same is true of every factor, so
\[
E_K(D)=\F_{h^m}.
\]
This proves \eqref{eq:componentdegrees}.

If the factors are linear, write one as $\ell(\Bf U)=\sum_i b_iU_i$ with $b_i\in\F_{h^n}$. The map $K^n\to\F_{h^n}$ given by $\ell$ is injective, since a nonzero vector in its kernel is a zero of $D$. Thus the $b_i$ form a $K$-basis of $\F_{h^n}$. By Lemma~\ref{lem:moore}, their conjugates give independent linear forms.
\end{proof}

The argument based on the Frobenius action on the absolute factors also underlies \cite[Proposition~16]{Menichetti96}.
In particular, the component splitting field alone cannot distinguish two $n$-dimensional presemifields whose determinant forms have only linear absolute factors: in both cases it is $\F_{h^n}$. In the calculations below, we transfer split factors by \eqref{eq:factor-coordinate-transfer} and use \eqref{eq:split-factor-frobenius} to determine their minimal fields of definition in the ordinary coordinates.

Suppose that
\begin{equation}\label{eq:simplexdet}
 D_{L,\Ps}(\Bf U)=\lambda\ell_0(\Bf U)\ell_1(\Bf U)\cdots\ell_{n-1}(\Bf U)
\end{equation}
over $E=E_{L,K}(\Ps)$, with independent linear factors. The hyperplanes $H_i=\{\ell_i=0\}$ determine the \emph{determinant vertices}
\[
 v_i=\bigcap_{j\ne i}H_j\in\PG(\mathbb P\otimes_KE).
\]
One can first record their ranks in the multiset
\[
 \nu_{L,K}(\Ps)
 =\multiset{\rank_E A_{L,\Ps}(v_i):0\le i<n}.
\]
The representative of $v_i$ in this expression is a nonzero vector in the ordinary coordinate system in which the forms $\ell_j$ are written; rescaling it does not change the rank.

For a presemifield this multiset has only one value. By Proposition~\ref{prop:splittingfield}, Frobenius acts transitively on the component hyperplanes and therefore also on their opposite vertices. The pencil, in its original coordinates, is defined over $K$. Applying Frobenius to its parameter and entries preserves matrix rank. Hence
\[
 \rank_E A_{L,\Ps}(v_0)=\cdots=\rank_E A_{L,\Ps}(v_{n-1}).
\]
The following definition records this common value.

\begin{definition}\label{def:vertex}
Let $d_{\mu,K}$ be the common degree of the absolute factors of $D_{\mu,\Ps}$, and write $d_\mu=d_{\mu,K}$ when the scalar field is clear from the context. The \emph{extended determinant-vertex rank} of type $\mu$ is
\begin{equation}\label{eq:vertexinvariant}
 \rho_{\mu,K}(\Ps)=
 \begin{cases}
  \rank_{E_{\mu,K}(\Ps)}A_{\mu,\Ps}(v),&d_\mu=1,\\
  \bot,&d_\mu>1,
 \end{cases}
 \qquad \mu\in\{L,R,\mathrm{tr}\},
\end{equation}
where $A_{\mu,\Ps}$ is the ordinary matrix family defined in Section~\ref{subsec:determinants}, and $v$ is any nonzero ordinary-coordinate representative of a determinant vertex of the indicated type. When $d_\mu=1$, we call the numerical value its \emph{determinant-vertex rank}. The symbol $\bot$ means that the absolute components are nonlinear, so no determinant-vertex configuration exists. For this specified scalar field, put
\begin{equation}\label{eq:relative-profile}
 \mathcal I_{\mu,K}(\Ps)
 =\bigl([E_{\mu,K}(\Ps):K],\ d_{\mu,K},\ \rho_{\mu,K}(\Ps)\bigr).
\end{equation}
We call this the \emph{relative determinant profile}. For an arbitrary presemifield $\Ps$, choose an isotopic semifield $\Ss$ and identify its centre with $K_0=\F_{|Z(\Ss)|}$. The \emph{central determinant profile} is
\begin{equation}\label{eq:totalinvariant}
 \mathcal I_{\mu}(\Ps)
 :=\mathcal I_{\mu,K_0}(\Ss),
 \qquad K_0\cong Z(\Ss),\quad \mu\in\{L,R,\mathrm{tr}\}.
\end{equation}
\end{definition}

For calculations in split coordinates, the representative must be transformed together with the matrix family. By \eqref{eq:ordinary-split-matrices},
\begin{equation}\label{eq:vertex-coordinate-ranks}
 \begin{aligned}
 \rank A_{L,\Ps}(v)&=\rank L_{\Ps}(Mv),\\
 \rank A_{R,\Ps}(v)&=\rank R_{\Ps}(Mv),\\
 \rank A_{\mathrm{tr},\Ps}(v)&=\rank M_{\Ps}^{\mathrm{tr}}(Mv).
 \end{aligned}
\end{equation}
These equalities can be checked over any extension containing the vertex coordinates and the Moore-matrix entries, since further scalar extension does not change rank. Thus when a split determinant is a coordinate product, evaluation at $e_i$ means evaluation at a split vertex; its ordinary representative is $M^{-1}e_i$. Throughout, $E_{\mu,K}$, $\rho_{\mu,K}$ and $\mathcal I_{\mu,K}$ refer to the original presemifield with its specified scalar field, even when their values are calculated in split coordinates.

In the definition above, $\Ss$ is viewed with its natural $K_0$-bilinear multiplication. If the displayed multiplication of $\Ps$ is already $K_0$-bilinear, its profile can be computed directly without passing to $\Ss$. Independence of the choices in \eqref{eq:totalinvariant} follows from Theorem~\ref{thm:invariance}.

The first two entries of a relative profile multiply to the dimension over its stated base field. For the central profile, their product is the dimension over the full centre. In the linear-component case the new information is a single matrix rank. When the components are nonlinear, extending to the component splitting field does not make them linear; $\bot$ is a defined alternative, not an omitted calculation. By using the full centre, \eqref{eq:totalinvariant} removes the arbitrary choice of a smaller common scalar field. Proposition~\ref{prop:profile-restriction} below shows that profiles over all central subfields can be recovered from this one and the subfield degree. No analogous replacement by the three possibly unequal nuclei is made: that would require different parameter spaces and determinant definitions.

\begin{theorem}\label{thm:invariance}
For a fixed $K$, the relative profiles \eqref{eq:relative-profile} are preserved by isotopy between two presemifields whose multiplications are $K$-bilinear. The central profiles \eqref{eq:totalinvariant} are independent of the isotopic semifield and scalar identifications used to define them, and are invariants under isotopy of arbitrary finite presemifields.
\end{theorem}
\begin{proof}
For a $K$-linear isotopism $(A,B,C)$ from $\Ps$ to $\Qs$, Proposition~\ref{prop:splittingfield} preserves the fields and component degrees. Let $A_0,B_0,C_0$ be its ordinary coordinate matrices. The ordinary identity in Section~\ref{subsec:determinants} is
\[
 A_{L,\Qs}(A_0\Bf U)B_0=C_0A_{L,\Ps}(\Bf U).
\]
Its determinants show that $A_0$ maps the ordinary left component hyperplanes and their vertices to those for $\Qs$. At an ordinary vertex representative $v$, the same identity gives
\[
 \rank A_{L,\Qs}(A_0v)=\rank A_{L,\Ps}(v),
\]
because $B_0,C_0$ are invertible. Interchanging the two multiplication variables proves the right assertion. For the trace type, the invertible $K$-matrix $C_\vee$ defined in Section~\ref{subsec:determinants} maps the ordinary trace vertices, and
\[
 A_0^{\mathsf T}A_{\mathrm{tr},\Qs}(C_\vee v)B_0
 =A_{\mathrm{tr},\Ps}(v).
\]
This proves the trace rank assertion without assuming self-dual ordinary bases.

For a general isotopism between $K$-bilinear presemifields, Lemma~\ref{lem:semilinearity} and its dual give a common automorphism of $K$. It extends to every finite extension, preserves absolute factor degrees and matrix ranks, and maps each component field onto the unique extension of the same degree. Apply \eqref{eq:semilinear-reduction} and the $K$-linear assertion. Finally, any two semifields isotopic to the same presemifield have isomorphic centres. Identify both with $K_0$; their multiplications are $K_0$-bilinear. The relative assertion just proved gives equality of their central profiles, including independence of the chosen field identifications.
\end{proof}

\begin{proposition}[Restriction of the scalar field]\label{prop:profile-restriction}
Let $k\subseteq K$ be finite fields, with $[K:k]=a$, and let $\Ps$ have $K$-bilinear multiplication. If
\[
 \mathcal I_{\mu,K}(\Ps)=(m,d,\rho),
\]
then
\begin{equation}\label{eq:profile-restriction}
 \mathcal I_{\mu,k}(\Ps)=(am,d,\rho),\qquad
 E_{\mu,k}(\Ps)=E_{\mu,K}(\Ps)
\end{equation}
as subfields of a common algebraic closure. Here $\rho=\bot$ remains $\bot$.
\end{proposition}
\begin{proof}
Put $n=\dim_K\mathbb P$, and extend scalars to an algebraic closure $\Omega$ of $k$. Since finite-field extensions are separable,
\[
 K\otimes_k\Omega\cong\prod_{j=0}^{a-1}\Omega,
 \qquad u\otimes v\longmapsto(\tau_j(u)v)_j,
\]
where the $\tau_j$ are the $k$-embeddings of $K$ into $\Omega$. Accordingly, the scalar extension of the underlying $k$-space is the direct sum of $a$ spaces of dimension $n$, one for each $\tau_j$. The $K$-bilinearity of multiplication makes products between different summands zero, and the product on summand $j$ is obtained by applying $\tau_j$ to the original structure constants. For the trace form, identify the additive space over $K$ with $F=\F_{|K|^n}$ as in Section~\ref{subsec:determinants}. The same decomposition then follows from
\[
 \Tr_{F/k}(z(x*y))
 =\Tr_{K/k}\!\left(\Tr_{F/K}(z(x*y))\right).
\]
Thus after compatible coordinate changes over $\Omega$, each ordinary multiplication pencil over $k$ is block diagonal with $a$ conjugate $n\times n$ blocks; the trace pencil has this form after congruence. Its determinant may acquire a nonzero scalar, which does not affect factorization or ranks. The variable sets of the blocks are disjoint.

Each block determinant has $m$ distinct absolutely irreducible factors of degree $d$. Factors belonging to different blocks cannot be proportional because they involve disjoint variable sets. There are consequently $am$ absolute factors of the $k$-determinant, all of degree $d$. If $|k|=h$, Proposition~\ref{prop:splittingfield} gives the component field $\F_{h^{am}}$ over either base, since $|K|=h^a$.

If $d=1$, a determinant vertex of the full block matrix has a nonzero parameter only in one block. In that block it is a determinant vertex for the original pencil, and every other block is zero. Its matrix rank is therefore $\rho$. If $d>1$, there are no determinant vertices over either base. This proves \eqref{eq:profile-restriction}.
\end{proof}

For example, a field has central profile $(1,1,1)$. Its relative profile as an $n$-dimensional algebra over a proper subfield is $(n,1,1)$, recovered by the proposition. The formulas below retain the base fields used to write the classical multiplications; when that field is smaller than the centre, they are relative profiles in the sense of \eqref{eq:relative-profile}.

\begin{lemma}\label{lem:monomial}
Let $\Ps=(\mathbb P,+,*)$ and $\Qs=(\mathbb Q,+,\circ)$ be presemifields of dimension $n$ with $K$-bilinear multiplications. Suppose that in the specified split coordinates over a common extension $E/K$, each polynomial $\widetilde D_{\mu,\Ps}$ and $\widetilde D_{\mu,\Qs}$, $\mu\in\{L,R,\mathrm{tr}\}$, is a nonzero scalar multiple of the product of the $n$ parameter coordinates. For every $K$-linear isotopism $(A,B,C)$ from $\Ps$ to $\Qs$, the matrices of $A,B,C$ in these split coordinates are monomial.
\end{lemma}
\begin{proof}
Write $A_{\mathrm{sp}},B_{\mathrm{sp}},C_{\mathrm{sp}}$ for the three split matrices. Taking determinants in \eqref{eq:pencilisotopy} gives
\[
 \prod_i(A_{\mathrm{sp}}\Bf X)_i=\kappa\prod_iX_i,\qquad \kappa\ne0.
\]
Unique factorization forces each row form of $A_{\mathrm{sp}}$ to be a scalar multiple of one coordinate variable. Invertibility makes those variables distinct, so $A_{\mathrm{sp}}$ is monomial. The right determinant gives the same conclusion for $B_{\mathrm{sp}}$. The trace determinant gives it for $C_{\mathrm{sp}}^{-\mathsf T}$, and therefore also for $C_{\mathrm{sp}}$.
\end{proof}

\subsection{Finite fields and Albert's generalized twisted fields}\label{subsec:examples}
We use $K=\F_h$ as the base field.

For the field $\Fs=(\F_{h^n},+,\cdot)$, all three split matrices are diagonal, and $\widetilde D_{\mu,\Fs}$ is the coordinate product for each type $\mu$. Transferring these factors by \eqref{eq:factor-coordinate-transfer} and applying Proposition~\ref{prop:splittingfield} gives
\begin{equation}\label{eq:fieldranks}
 E_{L,K}(\Fs)=E_{R,K}(\Fs)=E_{\mathrm{tr},K}(\Fs)=\F_{h^n},\qquad
 \rho_{L,K}(\Fs)=\rho_{R,K}(\Fs)=\rho_{\mathrm{tr},K}(\Fs)=1.
\end{equation}

Let $F=\F_{h^n}$, let $0\le s,t<n$, and let $c\in F^\times$.
Consider the multiplication
\begin{equation}\label{eq:gtf}
	x\circ y=xy-cx^{h^t}y^{h^s}.
\end{equation}
Put $\kappa=\gcd(n,s,t)$. Then $\Qs=(F,+,\circ)$ is a
presemifield if and only if
\begin{equation}\label{eq:gtf-division}
	\N_{F/\F_{h^\kappa}}(c)
	=c^{(h^n-1)/(h^\kappa-1)}\ne1.
\end{equation}
This is the finite-field norm formulation of Albert's
division condition \cite[Equation~(5)]{Albert}.

Indeed, for $x,y\in F^\times$, the equality $x\circ y=0$
is equivalent to
\[
c=x^{1-h^t}y^{1-h^s}.
\]
If $F^\times=\langle\xi\rangle$, the set of values on the
right is the subgroup
\[
\begin{aligned}
	\{x^{1-h^t}y^{1-h^s}:x,y\in F^\times\}
	&=
	\left\langle
	\xi^{\gcd(h^n-1,h^t-1,h^s-1)}
	\right\rangle\\
	&=
	\langle\xi^{h^\kappa-1}\rangle
	=
	\ker\N_{F/\F_{h^\kappa}}.
\end{aligned}
\]
This proves \eqref{eq:gtf-division}, since the multiplication
is $\F_h$-bilinear.

If $s=0$ or $t=0$, the multiplication factors through an invertible linear map in one variable and is isotopic to field multiplication. If $s=t$, it factors through an invertible linear map on the output and is again isotopic to field multiplication. We therefore restrict the following calculation to
\[
s,t,s-t\ne0\pmod n.
\]

For $s,t,s-t\ne0\pmod n$, Albert's calculation \cite[Theorem~1]{Albert} shows that the centre of an isotopic semifield has order $h^\kappa$. Thus $K=\F_h$ is the full centre precisely when $\kappa=1$. We retain arbitrary $\kappa$ below because the proposition computes relative profiles over the specified field $K$. To compute the central profiles instead, one may replace
\[
(h,n,s,t)
\quad\text{by}\quad
\left(h^\kappa,\frac n\kappa,
\frac s\kappa,\frac t\kappa\right).
\]

\begin{proposition}\label{prop:gtf}
Let $\Qs=(F,+,\circ)$ be defined by \eqref{eq:gtf}, where $F=\F_{h^n}$, $s,t,s-t\ne0\pmod n$, and \eqref{eq:gtf-division} holds. Put
\[
d=\gcd(n,s),\qquad \ell=n/d.
\]
\begin{enumerate}
\item If $d\mid t$, then
\[
 E_{L,K}(\Qs)=\F_{h^n},\qquad \rho_{L,K}(\Qs)=2,
 \qquad \mathcal I_{L,K}(\Qs)=(n,1,2).
\]
\item If $d\nmid t$, the left determinant has $d$ distinct absolutely irreducible components of degree $\ell\ge2$, and
\[
 E_{L,K}(\Qs)=\F_{h^d},\qquad
 \rho_{L,K}(\Qs)=\bot,\qquad
 \mathcal I_{L,K}(\Qs)=(d,\ell,\bot).
\]
\end{enumerate}
For the right determinant replace $(d,t)$ by $(\gcd(n,t),s)$. For the trace determinant replace it by $(\gcd(n,t-s),s)$. Thus the determinant-vertex rank of each type is two when the components are linear, and is $\bot$ otherwise.
\end{proposition}
\begin{proof}
The left and right determinant factorizations, together with the absolute irreducibility of their nonlinear components, are established in \cite[Lemmas~20 and~27, Proposition~28]{Menichetti96}. 

The split output forms are
\[
 g_r=X_rY_r-c^{h^r}X_{r+t}Y_{r+s}.
\]
In row $r$, the split left matrix $L_{\Qs}(\Bf X)$ therefore has the diagonal entry $X_r$ and the entry $-c^{h^r}X_{r+t}$ in column $r+s$. The permutation $r\mapsto r+s$ has $d$ cycles, each equal as a set to
\[
 I_r=\{r,r+d,\ldots,r+(\ell-1)d\},\qquad 0\le r<d.
\]
Reordering rows and columns by these cycles makes the matrix block diagonal. Within one block, choosing an off-diagonal entry occupies the diagonal column of the next row. That row must choose its off-diagonal entry as well, and the choice propagates around the entire cycle. Hence only the diagonal permutation and the full $\ell$-cycle contribute. The latter has sign
\[
 (-1)^{\ell-1}(-1)^\ell=-1.
\]
In the computation above, the first factor is the permutation sign, and the second comes from the $\ell$ negative matrix entries. This identity is also valid in every characteristic.

Put
\[
 P_r=\prod_{j=0}^{\ell-1}X_{r+jd},\qquad
 c_r=\prod_{j=0}^{\ell-1}c^{h^{r+jd}}.
\]
Their indices are periodic modulo $d$. The cycle calculation gives
\begin{equation}\label{eq:gtfdet}
 \widetilde D_{L,\Qs}(\Bf X)=\prod_{r=0}^{d-1}(P_r-c_rP_{r+t}).
\end{equation}
If $d\mid t$, then $I_{r+t}=I_r$ and $P_{r+t}=P_r$. The determinant becomes
\[
 \left(\prod_{r=0}^{d-1}(1-c_r)\right)\prod_{i=0}^{n-1}X_i.
\]
Since $d\mid t$, we have $\kappa=\gcd(n,s,t)=d$.
Moreover,
\[
c_0=\N_{F/\F_{h^d}}(c),
\qquad c_r=c_0^{h^r}.
\]
Thus \eqref{eq:gtf-division} gives $c_0\ne1$, and
\[
\prod_{r=0}^{d-1}(1-c_r)
=
\N_{\F_{h^d}/K}(1-c_0)\ne0.
\]
The absolute factors of $\widetilde D_{L,\Qs}$ are therefore the $n$ individual coordinate forms, not merely the $d$ cycle products. By \eqref{eq:factor-coordinate-transfer}, the ordinary determinant $D_{L,\Qs}(\Bf U)=\widetilde D_{L,\Qs}(M\Bf U)$ has $n$ independent absolute linear factors as well. Proposition~\ref{prop:splittingfield} now gives $E_{L,K}(\Qs)=\F_{h^n}$.

Suppose $d\nmid t$. Then $I_r$ and $I_{r+t}$ are disjoint. Choose $u\in I_r$ and write $P_r=X_uA$, $P_{r+t}=B$. Over
\[
 R=\overline K[X_i:i\ne u],
\]
the polynomial $AX_u-c_rB$ is primitive: $A$ and $B$ are monomials in disjoint variables, and $c_r\ne0$, so $\gcd(A,c_rB)=1$. It is linear and irreducible over $\operatorname{Frac}(R)$, and Gauss' lemma makes it irreducible over $R$. This proves absolute irreducibility of each split binomial. Substitution by $M$ preserves irreducibility and multiplicities by Lemma~\ref{lem:coordinate-transfer}. Proposition~\ref{prop:splittingfield}, applied to $D_{L,\Qs}$ in the original $K$-basis, therefore shows that $\widetilde D_{L,\Qs}$ has no repeated absolute factors. Thus the $d$ displayed binomials are pairwise nonproportional, each of degree $\ell$.

The component field can also be seen directly. If $b_0,\ldots,b_{n-1}$ is the original $K$-basis, write the split coordinates in terms of the ordinary variables as the linear forms
\[
 \ell_i(\Bf U)=(M\Bf U)_i=\sum_j b_j^{h^i}U_j.
\]
The ordinary factors corresponding to the binomials in \eqref{eq:gtfdet} are
\[
 G_r(\Bf U)=\prod_{j=0}^{\ell-1}\ell_{r+jd}(\Bf U)
 -c_r\prod_{j=0}^{\ell-1}\ell_{r+t+jd}(\Bf U).
\]
By \eqref{eq:split-factor-frobenius}, Frobenius on coefficients, fixing $\Bf U$, sends $\ell_i$ to $\ell_{i+1}$ and $c_r$ to $c_{r+1}$. It therefore cycles through the $d$ distinct ordinary components $G_r$, whose component splitting field is exactly $\F_{h^d}$. In the linear-factor case it cycles instead through the $n$ hyperplanes $\ell_i=0$, giving $\F_{h^n}$. The coefficients $c_r$ alone, with the split variables held fixed, do not determine this original component field.

In the linear-factor case, $e_j$ is a split determinant vertex, with ordinary representative $M^{-1}e_j$. Evaluate $L_{\Qs}$ at $\Bf X=e_j$. Only rows $j$ and $j-t$ are nonzero. They are distinct, and their nonzero entries occupy columns $j$ and $j-t+s$, which are distinct because $s\ne t$. Both coefficients are nonzero. Thus the split matrix rank is two. By \eqref{eq:vertex-coordinate-ranks}, $A_{L,\Qs}(M^{-1}e_j)$ has the same rank, proving the left assertion. Interchanging the inputs gives the right assertion.

For the trace determinant, let $\Qs^\dagger=(F,+,\circ^\dagger)$ be defined by
\begin{equation}\label{eq:gtf-adjoint}
 x\circ^\dagger z=xz-c^{h^{n-s}}x^{h^{t-s}}z^{h^{n-s}}.
\end{equation}
Invariance of trace under Frobenius gives
\[
 \Tr_{F/K}\bigl(z(x\circ y)\bigr)
 =\Tr_{F/K}\bigl(y(x\circ^\dagger z)\bigr).
\]
Hence the split matrices satisfy $M_{\Qs}^{\mathrm{tr}}(\Bf Z)=R_{\Qs^\dagger}(\Bf Z)^{\mathsf T}$, and $\widetilde D_{\mathrm{tr},\Qs}(\Bf Z)=\widetilde D_{R,\Qs^\dagger}(\Bf Z)$. With the same ordinary basis in both cases, \eqref{eq:ordinary-split-determinants} gives $D_{\mathrm{tr},\Qs}(\Bf W)=(\det M)^2D_{R,\Qs^\dagger}(\Bf W)$. The nonzero scalar has no effect on the component fields or ranks. The right-hand result applied to the shifts $t-s,-s$ gives $\gcd(n,t-s)$ and the condition that this divisor divides $s$. This proves all the assertions.
\end{proof}

In particular, if $n=r\ge5$ is prime and $\Qs$ is a proper Albert's presemifield, the three relevant greatest common divisors are all one. It follows that
\begin{equation}\label{eq:albert-prime}
 E_{L,K}(\Qs)=E_{R,K}(\Qs)=E_{\mathrm{tr},K}(\Qs)=\F_{h^r},\qquad
 \rho_{L,K}(\Qs)=\rho_{R,K}(\Qs)=\rho_{\mathrm{tr},K}(\Qs)=2.
\end{equation}
This calculation applies to all proper Albert's generalized twisted fields of that dimension. Their centre and all three nuclei also have order $h$, so a comparison with our family will require more than these nuclear parameters.

For an explicit nonlinear case, take $n=6$, $s=2$ and $t=1$, over any $\F_h$ with $h>2$, and choose $c$ with $\N_{\F_{h^6}/\F_h}(c)\ne1$. The split left determinant $\widetilde D_{L,\Qs}$ is a product of two absolutely irreducible cubics of the form
\[
 X_0X_2X_4-c_0X_1X_3X_5.
\]
After transfer to the ordinary coordinates, the component splitting field is $\F_{h^2}$, and the extended determinant-vertex rank is $\bot$. Neither cubic becomes a product of hyperplanes over any larger field. By contrast, the right and trace determinants split into six independent linear factors and have vertex ranks two. This gives a concrete example in which the three determinant types must be treated separately.

\subsection{Dickson semifields}
Let $K=\F_h$ have odd order, let $M=\F_{h^m}$ with $m>1$, and let $\sigma$ generate $\Gal(M/K)$. For a nonsquare $a\in M$, the Dickson semifield $\Ds=(M^2,+,\circ)$ has multiplication
\begin{equation}\label{eq:dickson-semifield}
 (x,y)\circ(u,v)=\bigl(xu+a y^\sigma v^\sigma,\ xv+yu\bigr);
\end{equation}
see \cite{LP,GK}. We use a product $K$-basis of $M^2$. The split coordinates in this subsection are obtained by applying the Moore-coordinate change of $M/K$ separately to its two coordinates.

\begin{samepage}
\begin{proposition}\label{prop:dickson}
Let $\Ds$ be the Dickson semifield in \eqref{eq:dickson-semifield}. Let $N_x=N_x(\Bf U_0)$ and $N_y=N_y(\Bf U_1)$ be the polynomial norms from $M$ to $K$ in the two ordinary coordinate vectors, and put $A=\N_{M/K}(a)$.
\begin{enumerate}
\item The ordinary left and right determinant forms over $K$ are
\begin{equation}\label{eq:dickson-det}
 \begin{aligned}
 D_{L,\Ds}(\Bf U_0,\Bf U_1)
 &=D_{R,\Ds}(\Bf U_0,\Bf U_1)\\
 &=N_x^2-AN_y^2.
 \end{aligned}
\end{equation}
Each has two distinct absolutely irreducible components of degree $m$, and
\[
 \begin{gathered}
 E_{L,K}(\Ds)=E_{R,K}(\Ds)=\F_{h^2},\qquad
 \rho_{L,K}(\Ds)=\rho_{R,K}(\Ds)=\bot,\\
 \mathcal I_{L,K}(\Ds)=\mathcal I_{R,K}(\Ds)=(2,m,\bot).
 \end{gathered}
\]
\item The ordinary trace determinant has $m$ distinct absolutely irreducible quadratic components, and
\[
 \begin{gathered}
 E_{\mathrm{tr},K}(\Ds)=\F_{h^m},\qquad
 \rho_{\mathrm{tr},K}(\Ds)=\bot,\\
 \mathcal I_{\mathrm{tr},K}(\Ds)=(m,2,\bot).
 \end{gathered}
\]
\end{enumerate}
\end{proposition}
\end{samepage}
\begin{proof}
We verify \eqref{eq:dickson-det} by an explicit block elimination. Write $\sigma(z)=z^{h^s}$, where $\gcd(m,s)=1$, and read all indices modulo $m$. Choose a $K$-basis $b_0,\ldots,b_{m-1}$ of $M$ and put
\[
 B_M=(b_j^{h^i})_{i,j},\qquad B_2=\diag(B_M,B_M).
\]
The ordinary parameter pair $(\Bf U_0,\Bf U_1)$ and the split parameter pair $(\Bf X_0,\Bf X_1)$ are related by $\Bf X_j=B_M\Bf U_j$ for $j=0,1$. We similarly use $(\Bf Y_0,\Bf Y_1)$ for the split coordinates of the input pair $(u,v)$. The split output forms of \eqref{eq:dickson-semifield} are
\[
 X_{0,i}Y_{0,i}+a^{h^i}X_{1,i+s}Y_{1,i+s},
 \qquad
 X_{1,i}Y_{0,i}+X_{0,i}Y_{1,i},
 \qquad 0\le i<m.
\]
With the input and output ordered by the first coordinate block and then the second, the split left matrix is
\[
 L_{\Ds}(\Bf X_0,\Bf X_1)
 =\begin{pmatrix}\Delta_0&C_s\\\Delta_1&\Delta_0\end{pmatrix},
\]
where
\[
 \Delta_0=\diag(X_{0,i}),\qquad
 \Delta_1=\diag(X_{1,i}),\qquad
 C_s=\diag(a^{h^i}X_{1,i+s})\Pi_s,
\]
and $(\Pi_s\Bf Y_1)_i=Y_{1,i+s}$. In particular, row $i$ of $C_s$ has its only nonzero entry $a^{h^i}X_{1,i+s}$ in column $i+s$.

Put
\[
 \mathscr R=M[X_{0,0},\ldots,X_{0,m-1},X_{1,0},\ldots,X_{1,m-1}],
 \qquad \mathscr E=\operatorname{Frac}(\mathscr R).
\]
Over $\mathscr E$, the diagonal matrix $\Delta_0$ is invertible. Multiplication on the left by
\[
 E=\begin{pmatrix}I_m&0\\-\Delta_1\Delta_0^{-1}&I_m\end{pmatrix}
\]
gives
\[
 EL_{\Ds}=\begin{pmatrix}\Delta_0&C_s\\0&S\end{pmatrix},
 \qquad S=\Delta_0-\Delta_1\Delta_0^{-1}C_s.
\]
Since $\det E=1$,
\[
 \det L_{\Ds}=\det\Delta_0\,\det S=\det(\Delta_0S).
\]
The matrices $\Delta_0$ and $\Delta_1$ are diagonal and commute, so
\[
 \Delta_0S
 =\Delta_0^2-\Delta_0\Delta_1\Delta_0^{-1}C_s
 =\Delta_0^2-\Delta_1C_s.
\]
Thus
\begin{equation}\label{eq:dickson-reduced-matrix}
 \begin{aligned}
 \widetilde D_{L,\Ds}(\Bf X_0,\Bf X_1)
 &=\det L_{\Ds}(\Bf X_0,\Bf X_1)\\
 &=\det\!\left(\diag(X_{0,i}^2)
 -\diag(a^{h^i}X_{1,i}X_{1,i+s})\Pi_s\right).
 \end{aligned}
\end{equation}
Only the commutativity of $\Delta_0$ and $\Delta_1$ was used here; neither is assumed to commute with $C_s$.

Since $\gcd(m,s)=1$, the same cycle determinant calculation as in the proof of Proposition~\ref{prop:gtf} gives
\[
 \begin{aligned}
 \widetilde D_{L,\Ds}(\Bf X_0,\Bf X_1)
 &=\prod_{i=0}^{m-1}X_{0,i}^2
 -\prod_{i=0}^{m-1}\bigl(a^{h^i}X_{1,i}X_{1,i+s}\bigr)\\
 &=\left(\prod_{i=0}^{m-1}X_{0,i}\right)^2
 -A\left(\prod_{i=0}^{m-1}X_{1,i}\right)^2,
 \end{aligned}
\]
where $\prod_i a^{h^i}=A$ and $\prod_iX_{1,i+s}=\prod_iX_{1,i}$. Although the block elimination was performed in $\mathscr E$, both sides belong to $\mathscr R$, so this is a polynomial identity. It remains valid when some $X_{0,i}$ are zero.

The ordinary and split matrices are related by the similarity in \eqref{eq:ordinary-split-matrices}, with $B_2$ in place of $M$. Hence
\[
 D_{L,\Ds}(\Bf U_0,\Bf U_1)
 =\widetilde D_{L,\Ds}(B_M\Bf U_0,B_M\Bf U_1).
\]
Under this substitution, $\prod_iX_{0,i}$ and $\prod_iX_{1,i}$ become $N_x$ and $N_y$, respectively. Thus the ordinary determinant is $N_x^2-AN_y^2$. Commutativity makes the right multiplication matrix at the same parameter value equal to the left multiplication matrix, proving \eqref{eq:dickson-det}.

The element $A$ is a nonsquare in $K$, since the norm sends a generator of $M^\times$ to a generator of $K^\times$. Choose $b\in\F_{h^2}$ with $b^2=A$. In ordinary coordinates the two factors are $N_x-bN_y$ and $N_x+bN_y$. After the block Moore substitution, each is a difference of products in disjoint variable sets, and is absolutely irreducible by the same Gauss-lemma argument as in Proposition~\ref{prop:gtf}. Absolute irreducibility is preserved by the inverse substitution, so
\[
 E_{L,K}(\Ds)=E_{R,K}(\Ds)=\F_{h^2},
\]
and there are two absolute components of degree $m$. Thus $\rho_{L,K}(\Ds)=\rho_{R,K}(\Ds)=\bot$ for $m>1$. This also shows directly that \eqref{eq:dickson-semifield} has no zero divisors: at actual $x,y\in M$, the equality $N_x^2=AN_y^2$ implies $N_x=N_y=0$, hence $x=y=0$.

The trace determinant has a different factorization. Use the nondegenerate pairing
\[
 \langle(z_0,z_1),(v_0,v_1)\rangle
 =\Tr_{M/K}(z_0v_0+z_1v_1)
\]
on $M^2$. For fixed $(z_0,z_1)$, moving $\sigma$ across the trace shows that the bilinear form in the two input pairs is the trace of the $M$-bilinear form with matrix
\[
 \begin{pmatrix}
 z_0&z_1\\ z_1&(az_0)^{\sigma^{-1}}
 \end{pmatrix}.
\]
For the ordinary parameter vectors $\Bf W_0,\Bf W_1$, set $z_j=\sum_i b_iW_{j,i}$ and interpret conjugations and norms polynomially, as in \eqref{eq:polynomial-norm}. With $\gamma=(\det B_2)^2\in K^\times$, the ordinary trace determinant for the chosen pairing is
\begin{equation}\label{eq:dickson-trace}
 D_{\mathrm{tr},\Ds}(\Bf W_0,\Bf W_1)
 =\gamma\,\N_{M/K}\bigl(z_0(az_0)^{\sigma^{-1}}-z_1^2\bigr).
\end{equation}
Here $\gamma$ is the determinant of the Gram matrix of the pairing in the product basis. Equivalently, with $\Bf Z_j=B_M\Bf W_j$, the split determinant is
\[
 \widetilde D_{\mathrm{tr},\Ds}(\Bf Z_0,\Bf Z_1)
 =\prod_{i=0}^{m-1}
 \left(a^{h^{i-s}}Z_{0,i}Z_{0,i-s}-Z_{1,i}^2\right).
\]
Each factor is a nondegenerate quadratic form in three distinct variables. Since the characteristic is odd, it is absolutely irreducible. The block Moore change transfers these $m$ factors to the ordinary determinant; its nonzero congruence scalar $\gamma$ changes neither their degrees nor their fields of definition. Proposition~\ref{prop:splittingfield} therefore gives
\[
 E_{\mathrm{tr},K}(\Ds)=\F_{h^m},\qquad
 \rho_{\mathrm{tr},K}(\Ds)=\bot.
\]
Changing the chosen nondegenerate pairing is an invertible $K$-linear change in the ordinary trace variable and does not affect these invariants.
\end{proof}

Table~\ref{tab:families} records the invariant for each determinant type. The entries $L$, $R$ and $\mathrm{tr}$ refer respectively to the matrices obtained by fixing the first multiplication input, the second input, and the trace parameter, as in \eqref{eq:three-matrices}. When several types are listed, the row applies to each of them separately. For Albert's multiplication, put
\[
 (d_L,a_L)=(\gcd(n,s),t),\quad
 (d_R,a_R)=(\gcd(n,t),s),\quad
 (d_{\mathrm{tr}},a_{\mathrm{tr}})=(\gcd(n,t-s),s).
\]
In the rows labeled by $\mu$, choose one of these three pairs and apply the displayed divisibility condition. The column headed $E_{\mu,K}$ gives the component splitting field relative to the stated base field $K$. The degree column records the common degree of the absolute components, not the extension degree $[E_{\mu,K}:K]$. The last column is the single value in Definition~\ref{def:vertex}. The stated fields specify relative profiles; their relation to the full-centre profile is given in Proposition~\ref{prop:profile-restriction}.
\begin{table}[htbp]
\centering\small
\setlength{\tabcolsep}{3pt}
\caption{Component splitting fields, absolute component degrees and extended determinant-vertex ranks. Relative profiles over the stated scalar fields. The base field is $K=\F_h$.}
\label{tab:families}
\begin{tabularx}{\textwidth}{@{}>{\raggedright\arraybackslash}p{24mm}>{\raggedright\arraybackslash}Xcccc@{}}
\toprule
(Pre)semifield & Conditions & Determinant type & $E_{\mu,K}$ & Degree & $\rho$\\
\midrule
Field & dimension $n$ & $L,R,\mathrm{tr}$ & $\F_{h^n}$ & $1$ & $1$\\[4pt]
Albert & proper; $d_\mu\mid a_\mu$ & $\mu$ & $\F_{h^n}$ & $1$ & $2$\\[4pt]
Albert & proper; $d_\mu\nmid a_\mu$ & $\mu$ & $\F_{h^{d_\mu}}$ & $n/d_\mu$ & $\bot$\\[4pt]
Albert & proper; prime dimension $r$ & $L,R,\mathrm{tr}$ & $\F_{h^r}$ & $1$ & $2$\\[4pt]
Dickson & dimension $2m$, $m>1$ & $L,R$ & $\F_{h^2}$ & $m$ & $\bot$\\[4pt]
Dickson & dimension $2m$, $m>1$ & $\mathrm{tr}$ & $\F_{h^m}$ & $2$ & $\bot$\\
\bottomrule
\end{tabularx}
\end{table}

\FloatBarrier
\section{Structure and Isotopy of the New Family}\label{sec:structure}

Throughout this section the parameters satisfy \eqref{eq:admissible}. We first calculate the nuclei directly. This calculation does not use the determinant-vertex invariant or a classification of isotopy between members.

\subsection{Left, middle and right nuclei}\label{subsec:direct-nuclei}

Put $\mathscr C=\mathscr C(\Ps_{w,s})$ and $\mathscr C^d=\mathscr C(\Ps_{w,s}^d)$, with the right-multiplication convention of \eqref{eq:spreadset}. Write $m_a(x)=ax$ for ordinary field multiplication by $a\in F$.

\begin{theorem}\label{thm:direct-nuclei}
The idealisers of these spread sets are
\begin{equation}\label{eq:nuclear-idealisers}
 I_l(\mathscr C^d)=I_l(\mathscr C)=I_r(\mathscr C)
 =\{m_a:a\in K\}.
\end{equation}
All the idealisers are taken in $\End_{\F_p}(F)$.
\end{theorem}
\begin{proof}
Let $T\in I_l(\mathscr C^d)$. The parametrization $x\mapsto L_x$ is an additive isomorphism from $F$ onto $\mathscr C^d$. The inclusion $T\mathscr C^d\subseteq\mathscr C^d$ therefore gives a unique additive map $A$ with $T L_x=L_{Ax}$. In terms of multiplication,
\begin{equation}\label{eq:left-nuclear-identity}
 (Ax)*_{w,s}y=T(x*_{w,s}y)\qquad(x,y\in F).
\end{equation}
Put $N=en$. Express $A,T$ as $p$-polynomials, and group terms according to their exponent indices modulo $e$:
\[
 A=\sum_{r=0}^{e-1}A_r,\qquad T=\sum_{r=0}^{e-1}T_r.
\]
On the left of \eqref{eq:left-nuclear-identity}, every $y$-exponent has index divisible by $e$. On the right, the $y$-indices in $T_r(x*_{w,s}y)$ are congruent to $r$ modulo $e$. Uniqueness of the bilinearized polynomial gives $T_r(x*_{w,s}y)=0$ for $r\ne0$. Fixing $x\ne0$ and using surjectivity of $y\mapsto x*_{w,s}y$ proves $T_r=0$. Comparison of the $x$-indices now gives $A_r(x)*_{w,s}y=0$ for $r\ne0$, so $A_r=0$. Thus $A,T$ are $K$-linear, without assuming either is nonzero.

In the $\sigma$-ordering write
\[
 A(x)=\sum_{i\in\mathbb Z/n}a_i x^{\sigma^i},\qquad
 T(z)=\sum_{i\in\mathbb Z/n}c_i z^{\sigma^i}.
\]
Let $I=\{-2,-1,1,2\}\subset\mathbb Z/n\mathbb Z$. For $j\notin I$, the coefficient of $y^{\sigma^j}$ on the left of \eqref{eq:left-nuclear-identity} is zero. On the right it is
\begin{equation}\label{eq:nuclear-coefficient}
 \begin{split}
 &(c_{j-2}+c_{j+1})x^{\sigma^{j-1}}+c_{j+1}x^{\sigma^{j+2}}\\
 &\qquad+w(c_{j-1}+c_{j+2})x^{\sigma^{j+1}}
       +wc_{j-1}x^{\sigma^{j-2}}.
 \end{split}
\end{equation}
The four exponents are distinct. Hence $c_{j-2}=c_{j-1}=c_{j+1}=c_{j+2}=0$. If $c_i\ne0$, it follows that $i+I=I$. A translation stabilizing this four-element set has order dividing four, and its order also divides the odd integer $n$. It is therefore the zero translation. Thus $T=c_0\id$.

The coefficient of $y^{\sigma^2}$ gives $A(x)^\sigma=c_0x^\sigma$, so $A=m_a$ with $a^\sigma=c_0$. The coefficient of $y^{\sigma^{-2}}$ gives $wa^{\sigma^{-1}}=wc_0$, whence $a^{\sigma^2}=a$. Since $\sigma^2$ generates $\Gal(F/K)$, we have $a\in K$ and $c_0=a$. Conversely, $K$-bilinearity gives $m_a\mathscr C^d\subseteq\mathscr C^d$ for every $a\in K$. This determines $I_l(\mathscr C^d)$.

The identity $x*_{w,s}y=w(y*_{w^{-1},s}x)$ gives
\[
 \mathscr C(\Ps_{w,s})=m_w\mathscr C(\Ps_{w^{-1},s}^d).
\]
Their left idealisers are conjugate by $m_w$. The preceding calculation applies to $w^{-1}$, and $m_w$ commutes with every $K$-scalar map. Hence $I_l(\mathscr C)=\{m_a:a\in K\}$.

Finally, if $T\in I_r(\mathscr C)$, there is a unique additive map $B$ satisfying $R_yT=R_{By}$, or
\[
 (Tx)*_{w,s}y=x*_{w,s}(By).
\]
Take the absolute trace after multiplication by $z$, and write $T^*$ for the adjoint with respect to $\Tr_{F/\F_p}(uv)$. The cyclic trace identity \eqref{eq:cyclic-tensor} yields
\[
 \Tr_{F/\F_p}\bigl(x\,T^*(y*_{w,s}z)\bigr)
 =\Tr_{F/\F_p}\bigl(x\,((By)*_{w,s}z)\bigr).
\]
Nondegeneracy implies $T^*L_y=L_{By}$, so $T^*\in I_l(\mathscr C^d)$. Therefore $T^*=m_a$ for $a\in K$. Scalar multiplication is self-adjoint, giving $T=m_a$. Conversely, every $m_a$ with $a\in K$ belongs to $I_r(\mathscr C)$ by bilinearity. This completes \eqref{eq:nuclear-idealisers}.
\end{proof}

\begin{corollary}\label{cor:nuclei}
Every semifield $\Ss=(\mathbb S,+,\circ)$ isotopic to $\Ps_{w,s}$ has
\begin{equation}\label{eq:nuclear-signature}
 (|\mathbb S|,|Z(\Ss)|,|N_l(\Ss)|,|N_m(\Ss)|,|N_r(\Ss)|)
 =(q^n,q,q,q,q).
\end{equation}
The same holds throughout the Knuth orbit of $\Ps_{w,s}$.
\end{corollary}
\begin{proof}
The nuclear orders follow from Theorem~\ref{thm:direct-nuclei} and Proposition~\ref{prop:spread-nuclei}. Every element of $\mathscr C$ is $K$-linear, so its $K$-scalar maps satisfy the defining identities of $\mathscr Z_\omega(\mathscr C)$ for any $\omega\ne0$. Since $I_l(\mathscr C)$ consists of exactly these maps, the centre also has order $q$. Apply Corollary~\ref{cor:knuth} for the orbit assertion.
\end{proof}

\subsection{Determinant-vertex ranks}\label{subsec:new-invariants}

In the $\sigma$-ordered split coordinates, the determinant vertices are the coordinate points. The ordinary representative of $e_i$ is $M_\sigma^{-1}e_i$, where $M_\sigma=(b_j^{\sigma^i})_{i,j}$ for the fixed ordinary basis. Every entry of $L_w(e_0)$ outside the rows and columns $-2,-1,1,2$ is zero. On these rows and columns, in this order, the matrix is
\begin{equation}\label{eq:vertex-block}
 \begin{pmatrix}
 0&w&0&0\\
 1&0&1&0\\
 0&w&0&w\\
 0&0&1&0
 \end{pmatrix}.
\end{equation}
Its determinant is $w^2$, which is nonzero. Hence $L_w(e_0)$ has rank four. The same is true at every split coordinate point by cyclic translation and for every determinant type by Lemma~\ref{lem:cyclic}. By \eqref{eq:vertex-coordinate-ranks}, the corresponding ordinary vertices have the same matrix ranks. Theorem~\ref{thm:division} and Lemma~\ref{lem:cyclic} factor all three split determinants into $n$ independent coordinate forms. Transferring the factors to the ordinary coordinates and applying Proposition~\ref{prop:splittingfield} gives the component field $F=\F_{q^n}$ for every type. Theorem~\ref{thm:direct-nuclei} and Corollary~\ref{cor:nuclei} show that $K$ is the full centre of an isotopic semifield. Consequently,
\begin{equation}\label{eq:new-central-profile}
 \mathcal I_\mu(\Ps_{w,s})=(n,1,4),
 \qquad \mu\in\{L,R,\mathrm{tr}\}.
\end{equation}
Profiles over smaller central subfields add no information, by Proposition~\ref{prop:profile-restriction}.

\begin{theorem}\label{thm:nonisotopy}
No $\Ps_{w,s}$ satisfying \eqref{eq:admissible} is isotopic to a finite field or to an Albert's generalized twisted field.
\end{theorem}
\begin{proof}
A field of order $q^n$ has centre of order $q^n$, so Corollary~\ref{cor:nuclei} already excludes field isotopy. Suppose that $\Ps_{w,s}$ is isotopic to a proper Albert's generalized twisted field. Write its standard presemifield representative on $F=\F_{p^{en}}$ as
\[
 x\circ y=xy-cx^\tau y^\nu,
\]
with nontrivial distinct automorphisms $\tau,\nu$ of $F$. By Albert's nuclear calculation \cite[Theorem~1]{Albert}, its centre has the order of $\operatorname{Fix}(\tau)\cap\operatorname{Fix}(\nu)$. Corollary~\ref{cor:nuclei} forces this common fixed field to be $K=\F_q$. Both automorphisms therefore fix $K$, and the displayed multiplication is $K$-bilinear. In particular, it has the form in Proposition~\ref{prop:gtf} with $h=q$.

By Proposition \ref{prop:gtf}, its left determinant either has a nonlinear absolute component or has determinant-vertex rank two. Neither agrees with \eqref{eq:new-central-profile}. The invariance theorem, including its common field-automorphism twist, therefore excludes isotopy.
\end{proof}

The standing assumption $q\equiv1\pmod3$ also excludes the
Coulter--Matthews/Ding--Yuan semifields and their Knuth conjugates:
these have characteristic three, whereas $q=p^e\equiv1\pmod3$ implies
$p\ne3$; see \cite[Section~6]{LP}. 

In particular, for every prime $p\equiv1\pmod3$ and
every prime $n\ge5$, the construction gives semifields
of order $p^n$ with centre $\F_p$ that are not isotopic
to finite fields or generalized twisted fields.
Thus the construction provides an infinite collection
of prime-dimensional examples outside these two classes.

\subsection{Isotopisms within the family}\label{subsec:parameters}
We determine all isotopisms between members of the family as both the coefficient parameter $w$ and the generator parameter $s$ vary. We then count the isotopy classes, determine the autotopism groups, and prove that no member has a commutative isotope.

\begin{theorem}\label{thm:parameters}
Fix $q=p^e$ and $n$ satisfying \eqref{eq:admissible}. Let $w,v$ satisfy \eqref{eq:parameter}, and let $s,t$ be units modulo $n$. Every isotopism
\[
 (Ax)*_{v,t}(By)=C(x*_{w,s}y)
\]
is exactly one of the following, where $a,b\in K^\times$ and $0\le h<en$:
\begin{equation}\label{eq:all-isotopies}
\begin{aligned}
 t=s,\quad v=w^{p^h}:\qquad
 & A(x)=a x^{p^h},\quad B(x)=b x^{p^h},\quad C(x)=ab x^{p^h};\\
 t=-s,\quad v=w^{-p^h}:\qquad
 & A(x)=a x^{p^h},\quad B(x)=b x^{p^h},\quad C(x)=vab x^{p^h}.
\end{aligned}
\end{equation}
It is $K$-linear if and only if $e\mid h$.
\end{theorem}
\begin{proof}
First suppose the isotopism is $K$-linear. By Theorem~\ref{thm:division} and Lemma~\ref{lem:cyclic}, all three split determinant polynomials are nonzero scalar multiples of their coordinate products; the corresponding ordinary forms are scalar multiples of norm forms by \eqref{eq:ordinary-split-determinants} and \eqref{eq:polynomial-norm}. Lemma~\ref{lem:monomial} forces the input matrices and the inverse transpose of the output matrix to be monomial in their split coordinates. By \eqref{eq:dickson}, a $K$-linear map with a monomial Dickson matrix is a single linearized monomial. Hence
\[
 A(x)=a x^{q^i},\quad B(x)=b x^{q^j},\quad C(x)=c x^{q^k},
 \qquad a,b,c\in F^\times.
\]
In standard $q$-exponents, put
\[
 S_s=\{(s,2s),(s,-s),(-2s,-s),(2s,s),(-s,s),(-s,-2s)\}.
\]
Equality of supports gives
\[
 S_t+(i-k,j-k)=S_s.
\]
In the first-coordinate multiset of $S_s$, exactly $s,-s$ occur twice; $2s,-2s$ occur once. Thus, putting $u=i-k$, we get
\[
 \{u+t,u-t\}=\{s,-s\}.
\]
Adding gives $2u=0$, so $u=0$ since $n$ is odd, and $t=\pm s$. The same argument in the second coordinate gives $j=k$. Hence $i=j=k$ modulo $n$.

If $t=s$, comparison of the three coefficient-one terms gives
\[
 a^\sigma b^{\sigma^2}=a^\sigma b^{\sigma^{-1}}
 =a^{\sigma^{-2}}b^{\sigma^{-1}}=c.
\]
Therefore $a^{\sigma^3}=a$ and $b^{\sigma^3}=b$. Since $\gcd(n,3)=1$, both $a,b$ belong to $K$, and $c=ab$. Any coefficient-$w$ term now gives $v=w$. The converse is immediate by substitution. When $t=-s$, use $f_{v,-s}=v f_{v^{-1},s}$ to reduce to the first case; this gives $v=w^{-1}$ and $c=vab$.

For a general isotopism, apply Lemma~\ref{lem:semilinearity} to the right and left spread sets. If the companion automorphism on $K$ is the $p^j$-Frobenius, with $0\le j<e$, take $\phi_j(x)=x^{p^j}$ in \eqref{eq:galois-twist}. Since all Frobenius powers commute,
\begin{equation}\label{eq:familytwist}
 \Ps_{w,s}^{[j]}=\Ps_{w^{p^j},s}.
\end{equation}
The preceding $K$-linear calculation applies to this twisted source. Combining its $q^k$-power with $p^j$ yields $h=j+ek$ and precisely \eqref{eq:all-isotopies}. Conversely all the displayed triples satisfy the isotopy identity. A monomial $x\mapsto a x^{p^h}$ is $K$-linear precisely when $e\mid h$.
\end{proof}

The number of choices for the generator parameter $s$ is $\varphi(n)=|(\mathbb Z/n\mathbb Z)^\times|$. Thus the totient in the class count is evaluated at the extension degree $n$, not at the base-field size $q$.

\begin{corollary}\label{cor:parametercount}
Fix $q=p^e$ and $n$ satisfying \eqref{eq:admissible}, and put $K=\F_q$. The family $\{\Ps_{w,s}\}$ represents exactly $\varphi(n)$ $K$-linear isotopy classes. The number of isotopy classes represented by this family is
\begin{equation}\label{eq:classcount}
 \begin{cases}
 \varphi(n),&p\equiv1\pmod3,\\
 \varphi(n)/2,&p\equiv2\pmod3.
 \end{cases}
\end{equation}
For $p\equiv1\pmod3$, fix either root $w$ and let $s$ range over all units modulo $n$ to obtain representatives. For $p\equiv2\pmod3$, fix $w$ and take one $s$ from each pair $\{s,-s\}$.
\end{corollary}
\begin{proof}
Under $K$-linear isotopy the only identification among the $2\varphi(n)$ parameter pairs is $(w,s)\sim(w^{-1},-s)$. The two roots are distinct, and $s\ne-s$ modulo the odd number $n$.

If $p\equiv1\pmod3$, both roots are in $\F_p$, so all $p$-Frobenius powers fix them. If $p\equiv2\pmod3$, they are exchanged by $p$-Frobenius and $e$ is even. In this case $(w,s)$ and $(w^{-1},s)$ are also isotopic. Theorem~\ref{thm:parameters} shows that these are all identifications.
\end{proof}

The case $n=5$, $s=1$ contains the two-parameter-value subfamily from which the construction arose. Allowing all generators adds genuinely different classes: for each admissible $q=p^e$, the full family of order $q^5$ has four isotopy classes when $p\equiv1\pmod3$ and two when $p\equiv2\pmod3$. More generally, for prime $n=r\ge5$ the two counts are $r-1$ and $(r-1)/2$, respectively; see Corollary~\ref{cor:menichetti}.

\begin{corollary}\label{cor:autotopisms}
Let $\delta=[\F_p(w):\F_p]\in\{1,2\}$. The autotopisms of $\Ps_{w,s}$ are
\[
 (a x^{p^h},b x^{p^h},ab x^{p^h}),
 \qquad a,b\in K^\times,\quad \delta\mid h,
\]
with exponents modulo $en$. In particular,
\begin{equation}\label{eq:autorder}
 |\Aut(\Ps_{w,s})|=(q-1)^2en/\delta,
 \qquad
 \Aut(\Ps_{w,s})\cong(K^\times\times K^\times)\rtimes C_{en/\delta},
\end{equation}
where the cyclic generator acts on each scalar by $p^\delta$-Frobenius.
\end{corollary}
\begin{proof}
In a self-isotopy $t=s$, since $s\ne-s$. The condition on $h$ is $w^{p^h}=w$, equivalently $\delta\mid h$. Composition of the displayed maps gives the stated semidirect product.
\end{proof}

For any fixed candidate $\Qs$ with $K$-bilinear multiplication, \eqref{eq:semilinear-reduction} and \eqref{eq:familytwist} show that isotopy to $\Ps_{w,s}$ reduces to $K$-linear isotopy with one or two parameter values. For $p\equiv1\pmod3$ only $w$ occurs, so isotopy and $K$-linear isotopy are equivalent for that fixed member. For $p\equiv2\pmod3$ both $w,w^{-1}$ must be considered.

\begin{corollary}\label{cor:no-commutative}
No member of the Knuth orbit of $\Ps_{w,s}$ has a commutative isotope.
\end{corollary}
\begin{proof}
By Ganley's criterion, a commutative isotope would give an additive bijection $H$ with $x*_{w,s}(Hy)=y*_{w,s}(Hx)$. Replacing $y$ by $H^{-1}y$ and using \eqref{eq:opposite} yields
\[
 (Hx)*_{w^{-1},s}(H^{-1}y)=w^{-1}(x*_{w,s}y).
\]
Its output map is $K$-linear. Common semilinearity therefore makes $H$ $K$-linear too. The $K$-linear part of Theorem~\ref{thm:parameters}, with the same $s$, would require $w^{-1}=w$, a contradiction. Corollary~\ref{cor:knuth} proves the statement throughout the orbit.
\end{proof}

\subsection{Kaplansky's conjecture and the prime-dimensional classification}\label{subsec:menichetti}

Kaplansky's five-dimensional conjecture \cite{Kaplansky75}, as reproduced in \cite[p.~69, (K2)]{Menichetti96}, asserts that sufficiently large base fields force a field or a twisted field. Menichetti's assertion in prime dimension would answer it affirmatively. We first record a concrete failure of the rank implication on page~91 of \cite{Menichetti96}. This failure already occurs within a generalized twisted isotopy class, so it does not depend on the present construction. Work over $F=\F_{125}$ with $K=\F_5$ and $\sigma(x)=x^5$. Let $\Qs=(F,+,\circ)$ and $\Ps_0=(F,+,\star)$ be defined by
\[
 x\circ y=xy-2x^\sigma y^{\sigma^2},\qquad
 x\star y=(\id+\sigma)(x\circ y).
\]
Their split output forms are
\[
 q_r=X_rY_r-2X_{r+1}Y_{r+2},\qquad p_r=q_r+q_{r+1}
 \quad(r\bmod3).
\]
The presemifield $\Qs$ is a generalized twisted representative, since $\N_{\F_{125}/\F_5}(2)=2^3\ne1$. The output change $\id+\sigma$ is invertible, with inverse $(\id-\sigma+\sigma^2)/2$. The split determinants are $\widetilde D_{L,\Ps_0}(\Bf X)=-14\prod_iX_i$ and $\widetilde D_{R,\Ps_0}(\Bf Y)=-14\prod_iY_i$. Yet
\[
 p_0=X_1(Y_1-2Y_2)+Y_0(X_0-2X_2).
\]
At the split parameter $(Y_0,Y_1,Y_2)=(0,2,1)$ the first equation vanishes identically, while the coefficient matrix in the $X$-variables is
\[
 \begin{pmatrix}0&0&0\\-4&2&1\\-4&-2&1\end{pmatrix}.
\]
Its bottom-left $2\times2$ minor is $16\ne0$ in $\F_5$, so its rank is two, not at most one. A zero equation does not, by itself, force the asserted extra rank drop. This illustrates the gap without contradicting Menichetti's three-dimensional classification \cite{Menichetti77}.

We now specialize the construction to prime dimension while retaining every admissible prime-power base field. The resulting semifields contradict the classification in every prime dimension at least five; dimension five gives counterexamples to Kaplansky's conjecture.
\begin{corollary}\label{cor:menichetti}
Let $q=p^e\equiv1\pmod3$ be a prime power, and let $r\ge5$ be prime. The family $\Ps_{w,s}$ with $n=r$ represents exactly
\[
 \begin{cases}
 r-1,&p\equiv1\pmod3,\\[2mm]
 (r-1)/2,&p\equiv2\pmod3
 \end{cases}
\]
isotopy classes of semifields of order $q^r$. Every such semifield has centre and all three nuclei of order $q$, and none is isotopic to a finite field or an Albert's generalized twisted field. The five-dimensional examples disprove Kaplansky's conjecture, while the examples in every prime dimension $r\ge5$ contradict Proposition~30 and Corollary~33 of \cite{Menichetti96} as stated.
\end{corollary}
\begin{proof}
Every prime $r\ge5$ is coprime to $6$. Apply Theorem~\ref{thm:division} and Lemma~\ref{lem:kaplansky} with $n=r$. Corollary~\ref{cor:nuclei} gives the nuclear and central orders, and Theorem~\ref{thm:nonisotopy} excludes fields and generalized twisted fields. Since $\varphi(r)=r-1$, Corollary~\ref{cor:parametercount} gives the stated numbers of isotopy classes.

For fixed $r$ and fixed characteristic $p\ne3$, arbitrarily large admissible base fields are obtained by taking $q=p^e$ with all $e\ge1$ if $p\equiv1\pmod3$, and with all even $e$ if $p\equiv2\pmod3$. The left and right norm determinants are exactly the hyperplane situation required in Proposition~30 of \cite{Menichetti96}. Hence both that proposition and the prime-dimensional classification in Corollary~33 are contradicted, and no increase of the base-field threshold repairs those assertions. Taking $r=5$ gives the stated consequence for Kaplansky's conjecture.
\end{proof}

\section*{Acknowledgements}
Gabor Nagy is supported by the National Research, Development and Innovation Fund of Hungary under grant SNN 152582. Yue Zhou is supported by the National Natural Science Foundation of China (No.\ 12371337).

The main part of this paper was developed during the conference \emph{Finite Geometry, Combinatorics, Boolean Functions and their Links}, held in Magdeburg, Germany, on 15--19 September 2026 and dedicated to Alexander Pott on the occasion of his 65th birthday. The authors thank the organizers for providing a stimulating setting and a welcoming atmosphere for discussions of open problems.

Yue Zhou identified gaps in Menichetti's paper and communicated his concerns to J\"urgen Bierbrauer by email on 8 May 2014. The same gaps were also identified by Guobiao Weng, who shared his findings and principal concerns with Tao Feng, Qing Xiang and Yue Zhou in 2021. Yue Zhou thanks Bierbrauer, Feng, Xiang and Weng for valuable discussions.

\section*{Declaration of generative AI and AI-assisted technologies in the manuscript preparation process}
During the research and preparation of this work, the authors used ChatGPT~6 Pro (OpenAI) under their guidance and through sustained mathematical discussions. Most of the results were developed with its assistance, including the generation of candidate constructions, proposed identities and proof strategies, draft arguments and verification code. The authors reviewed and edited the generated material and take full responsibility for the mathematical claims, references and final content of the paper.  The results have been formalized by Aristotle (Harmonic) and verified by Lean 4; see \cite{LeanRepo}.

\end{document}